\documentclass[11pt]{article}
\usepackage{graphicx}
\usepackage{bm}
\usepackage{amsmath,amsthm,amssymb}
\usepackage{multirow,bigdelim}
\usepackage{amscd}
\usepackage{multirow}

\newtheorem{thm}{Theorem}[section]
\newtheorem{theorem}[thm]{Theorem}

\newtheorem{lemma}[thm]{Lemma}

\newtheorem{claim}[thm]{Claim}
\newtheorem{corollary}[thm]{Corollary}

\DeclareMathOperator{\rank}{rank}

\DeclareMathOperator{\spa}{span}

{%
   \begin{list}{$\bullet$\ \ }
   {%
      \setlength{\itemindent}{0pt}
      \setlength{\leftmargin}{3zw}
       \setlength{\rightmargin}{0zw}
       \setlength{\labelsep}{0zw}
       \setlength{\labelwidth}{3zw}
       \setlength{\itemsep}{0em}
       \setlength{\parsep}{0em}
       \setlength{\listparindent}{0zw}
   }
}{%
   \end{list}%
}

\begin{document}
\title{The Signed Positive Semidefinite Matrix Completion Problem for Odd-$K_4$ Minor Free Signed Graphs}
\author{Shin-ichi Tanigawa\thanks{
Centrum Wiskunde $\&$ Informatica (CWI), Postbus 94079, 1090 GB
Amsterdam, The Netherlands, and
Research Institute for Mathematical Sciences,
Kyoto University, Sakyo-ku, Kyoto 606-8502, Japan.
e-mail: {\tt tanigawa@kurims.kyoto-u.ac.jp}}}
\maketitle

\begin{abstract}
We give a signed generalization of Laurent's theorem that characterizes feasible positive semidefinite matrix completion problems in terms of metric polytopes. 
Based on this result, we give a characterization of the maximum rank completions of the signed  positive semidefinite matrix completion problem for  odd-$K_4$ minor free signed graphs. 
The analysis can also be used to bound the minimum rank over the completions and to characterize uniquely solvable completion problems for odd-$K_4$ minor free signed graphs. 
As a corollary we derive a characterization of the universal rigidity of odd-$K_4$ minor free spherical tensegrities, and also a  characterization of signed graphs whose signed Colin de Verdi{\`e}re parameter $\nu$ is bounded by two, recently shown by Arav et al.

\end{abstract}

\section{Introduction}
Given an undirected graph $G=(V,E)$ with vertex set $V=\{1,\dots, n\}$ and an edge weight $c:E\rightarrow [-1,1]$,   the {\em (real) positive semidefinite  matrix completion problem} ${\rm P}(G,c)$ asks to decide whether the following set is empty, and if not  to find a point in it:
\[
\left\{X\in{\cal S}^n_+\mid X[i,i]=1 \ \forall i\in V,  \  X[i,j]=c(ij) \ \forall ij\in E \right\}
\]
where ${\cal S}^n_+$ denotes the set of positive semidefinite matrices of size $n$.
If $G$ has no edge, then its feasible region is the set ${\cal E}_n:=\{X\in {\cal S}^n_+\mid X[i,i]=1 \ \forall i\in V\}$ of correlation matrices, which is known as the {\em elliptope}.
In general  the  set ${\cal E}(G)$ of edge weights $c$ for which the program ${\rm P}(G,c)$ is feasible is 
the projection of the elliptope along the coordinate axes~\cite{l97}, 
and understanding ${\cal E}(G)$ is one of fundamental questions in this context (see \cite{bjl,dl,l98, gjsw}).
It has been shown by Laurent~\cite{l97} that ${\rm arccos}({\cal E}(G))/\pi$ coincides with the {\em metric polytope} of a graph $G$  if and only if $G$ is $K_4$-minor free. The definition of the metric polytope will be given in the next section.

In this paper we consider a signed version of the PSD matrix completion problem.
A {\em signed graph} is a pair $(G,\Sigma)$ of an undirected graph $G$ (which may contains parallel edges) and $\Sigma\subseteq E$. For a signed graph $(G,\Sigma)$ and $c:E\rightarrow  [-1,1]$, 
the {\em signed PSD matrix completion problem} ${\rm P}(G,\Sigma,c)$ asks to decide  wether the following set is empty and to find a point if it is nonempty:
\[
\left\{ X\in{\cal S}^n_+ \mid 
X[i,i]=1 \ \forall i\in V, \  
X[i,j]\geq c(ij) \  \forall ij\in E\setminus \Sigma, \ 
X[i,j]\leq c(ij) \  \forall ij\in \Sigma
\right\}. 
\]
We prove a signed generalization of Laurent's theorem which says that,
 denoting by ${\cal E}(G,\Sigma)$ the set  of edge weights $c$ for which ${\rm P}(G,\Sigma, c)$ is feasible,
 ${\rm arccos}({\cal E}(G, \Sigma))/\pi$ coincides with a signed version of the metric polytope if and only if $(G, \Sigma)$ is odd-$K_4$ minor free.
(See Section~\ref{subsec:signed} for the definition of minors of signed graphs.)
In fact our main theorem (Theorem~\ref{thm:oddK4}) states  a much stronger property of the completion problem.
Namely, if $(G,\Sigma)$ is odd-$K_4$ minor free and $c\in [-1, 1]^E$ is nondegenerate\footnote{We say that $c$ is {\em nondegenerate} if $c\in (-1, 1]^{\Sigma}\times [-1,1)^{E\setminus \Sigma}$.
Note that, if $c(ij)=1$ with $ij\in E\setminus \Sigma$, then $X[i,j]=1$ and $X[i,k]=X[j,k]$ for any feasible $X$ and any $k\in V(G)\setminus \{i,j\}$. This implies that the feasible set of ${\rm P}(G,c)$ is equal to that of ${\rm P}(G/ij, c)$, where $G/ij$ is the graph obtained from $G$ by contracting $ij$, and we can always  focus on the contracted smaller problem. (In the Euclidean matrix completion problem, such a degeneracy corresponds to the case when  the distance between $i$ and $j$ is specified to be  zero, in that case $i$ and $j$ being recognized as just one point.) A similar trivial reduction is  possible if $c(ij)=-1$ with $ij\in \Sigma$. Hence, from the practical view point we can always assume that $c$ is nondegenerate.}, the strict complementarity always holds in ${\rm P}(G, \Sigma, c)$, 
and the maximum rank  is characterized by the rank of a dual solution determined by the facet of ${\rm arccos}(c)/\pi$ in the signed metric polytope.

In the proof we also obtain that any feasible ${\rm P}(G,\Sigma, c)$ has a solution of rank at most three (resp., at most two) if $(G,\Sigma)$ is odd-$K_4$ minor free (resp., odd-$K_3^2$ minor and odd-$K_4$ minor free).
This is the first signed version of the low-dimensional embeddability of edge-weighted graphs into the spherical space or  the Euclidean space discussed in~\cite{b07,bc07,lv14,h08}. 
Our main theorem will also play a key role in the analysis of the {\em singularity degree} of the positive semidefinite matrix completion problem in \cite{singularity}.


It is known that the strict complementarity condition is closely related to the unique solvability in semidefinite programming~\cite{cg}. 
Adapting the analysis,  we also give a characterization of uniquely solvable signed PSD matrix completions for odd-$K_4$ minor signed graphs (Theorem~\ref{thm:unique_K4}).
The characterization can be tested in polynomial time by a repeated application of a shortest path algorithm, provided that ${\rm arccos(c)}/\pi$ is given as input.
The concept of unique solvability of the signed PSD matrix completion problem coincides with the so-called {\em universal rigidity} of {\em spherical tensegrities} in rigidity theory.
(A tensegrity is a structure made of cables and struts, see, e.g.,~\cite{c14}.)
Thus our unique solvability characterization implies a characterization of the universal rigidity of odd-$K_4$ minor free spherical tensegrities (Corollary~\ref{thm:universal_K4}).
The universal rigidity is a modern topic in rigidity theory, which was introduced by Zhu, So, and Ye~\cite{zsy} and was implicit in a classical paper by Connelly~\cite{c}. 
Its characterization is known in the unsigned  {\em generic} case~\cite{gt}, and 
understanding it at the level of graphs or/and in nongeneric cases is recognized as a challenging problem~\cite{cg15,gt,jh}.

Our characterization of universal rigidity has an application to a graph parameter, 
{\em Colin de Verdi{\`e}re parameter} $\nu(G)$,  introduced by Colin de Verdi{\`e}re~\cite{c98}.
This  is one of well-studied parameters among those defined in terms of spectral properties  of graphs (see, e.g., \cite{h96,h02}). 
Recently Arav et al.~\cite{ahlv13} introduced a signed version of the Colin de Verdi{\`e}re parameter and gave a characterization of signed graphs for which the signed Colin de Verdi{\`e}re parameter is equal to one. Later in \cite{ahlv16} they further gave a characterization of signed graphs whose  signed Colin de Verdi{\`e}re parameter is bounded by at most two. 
This characterization can be derived  as a corollary of our characterization of universal rigidity.


The paper is organized as follows. 
In Section~\ref{sec:pre} we introduce terminogies for proving our main theorem.
In Section~\ref{sec:main} we state our main theorem and a corollary,
and the proof  is given in Section~\ref{sec:proof}.
In Section~\ref{sec:unique} we discuss about uniquely solvable completion problems for odd $K_4$-minor signed graphs.
We give corollaries to  the universal rigidity of  spherical tensegrities in Section~\ref{sec:universal} and to signed Colin de Verdi{\'e}re parameter in Section~\ref{sec:colin}.

\section{Preliminaries}
\label{sec:pre}
In Section~\ref{subsec:completion} we formulate the signed PSD matrix completion problem and its dual.
In Section~\ref{subsec:signed} we introduce necessary notion about signed graphs and a structural theorem of odd-$K_4$ minor free signed graphs.
In Section~\ref{subsec:metric} we introduce the metric polytope and its signed version.

We use the following notation throughout the paper.
For an undirected graph $G$, $V(G)$ and $E(G)$ denote the vertex and the edge set of $G$, respectively.
If $G$ is clear from the context, we simply use $V$ and $E$ to denote $V(G)$ and $E(G)$, respectively.
For $F\subseteq E$, let $V(F)$ be the set of vertices incident to an edge in $F$.
For $X\subseteq V$, let $\delta(X)$ denotes the set of edges between $X$ and $V\setminus X$.
If $X=\{v\}$ for some $v\in V$, then $\delta(\{v\})$ is simply denoted by $\delta(v)$. 
A path $P$ is said to be {\em internally disjoint} from a graph $G$ if $P$ and $G$ are edge-disjoint and their vertices do not intersect except possibly at  endvertices of $P$.

For a finite set $X$,   let $\mathbb{R}^X$ be a $|X|$-dimensional vector space each of whose coordinate is associated with an element in $X$.

As given in the introduction, we denote ${\cal E}_n=\{X\in {\cal S}_+^n: \forall i, X[i,i]=1 \}$.
Each entry of a symmetric matrix of size $n$ is associated with an edge of the complete graph $K_n$.
Using this correspondence, the projection $\pi_G$ of the space of real symmetric matrices of size $n$ to $\mathbb{R}^E$ is defined.
Then ${\cal E}(G)=\pi_G({\cal E}_n)$.
Also let ${\bm e}_i$ be a vector in $\mathbb{R}^n$ whose $i$-th coordinate is one and the other entries are zero.

A positive semidefinite matrix $X$ of rank $d$ can be represented as $P^{\top}P$ for some $d\times n$ matrix of rank $d$. 
This representation is referred to as a {\em Gram matrix representation} of $X$.
By assigning the $i$-th column of $P$ with each vertex, one can obtain a map $p:V\rightarrow \mathbb{R}^d$. If $X\in {\cal E}_n$, $p$ is actually a map to the unit sphere $\mathbb{S}^{d-1}$.
Conversely, any $p:V\rightarrow \mathbb{S}^{d-1}$ defines $X\in {\cal E}_n$ of rank $d$ 
by $X[i,j]=p(i)\cdot p(j)$. This $X$ is denoted by ${\rm Gram}(p)$.

\subsection{SDP formulation}
\label{subsec:completion}

Given a signed graph $(G,\Sigma)$ and $c\in [-1, 1]^{E}$, we are interested in the following SDP, denoted by ${\rm P}(G,\Sigma,c)$, and its dual:
\begin{equation*}
\begin{array}{ccc}
\sup & 0 & \\
{\rm s.t.} & X[i,j]\geq c(ij) & (ij\in E\setminus \Sigma) \\
 & X[i,j]\leq c(ij) & (ij\in E\cap \Sigma) \\
 & X[i,i]=1 & (i\in V) \\
 & X\succeq 0
\end{array}
\quad
\begin{array}{cccc}
 \inf & \sum_{i\in V} \omega(i)+ \sum_{ij\in E} \omega(ij)c(ij)  & \\
{\rm s.t.} & \omega(ij)\leq 0 & (ij\in E\setminus \Sigma) \\
 & \omega(ij)\geq 0 & (ij\in E\cap \Sigma) \\
 & \sum_{i\in V} \omega(i) E_{ii}+ \sum_{ij\in E} \omega(ij)E_{ij}\succeq 0 \\
 & \omega\in \mathbb{R}^{V\cup E}
\end{array}
\end{equation*}
where $E_{ij}=({\bf e}_i{\bf e}_j^{\top}+{\bf e}_j{\bf e}_i^{\top})/2$. 
Throughout the paper, for $\omega\in \mathbb{R}^{V\cup E}$, we  use the capital letter $\Omega$ to denote $\sum_{i\in V} \omega(i) E_{ii}+ \sum_{ij\in E} \omega(ij)E_{ij}$.
We say that $\omega$ is {\em supported on} $F\subseteq E$ if $\omega(e)=0$ for every $e\in E\setminus F$.
The signed version of ${\cal E}(G)$ can be defined as 
\begin{equation*}
{\cal E}(G,\Sigma)=\{c\in [-1,1]^E \mid {\rm P}(G, \Sigma, c) \text{ is feasible} \}.
\end{equation*}

A pair $(X, \omega)$ of a primal and a dual feasible solutions are said to satisfy the {\em complementarity condition} if ${\rm Tr}(X \Omega)=0$ and $(X[i,j]-c(ij))\omega(ij)=0$ for every edge $ij\in E$.
Since the dual problem is strictly feasible,  this is equivalent to saying that $\omega$ is a dual optimal solution.
Also the complementarity condition implies $\rank X+\rank \Omega\leq |V|$. 
The pair is said to satisfy the {\em strict complementarity condition} if the inequality holds with equality.

The following lemma will be a fundamental tool to analyze the strict complementarity.
Essentially the same statement is given in \cite{c11}.
\begin{lemma}
\label{lem:glueing}
Let $(G,\Sigma)$ be a signed graph, 
and $(G_1,\Sigma_1)$ and $(G_2, \Sigma_2)$ be two signed subgraphs
with $E(G)= E(G_1)\cup E(G_2)$.
Let $c\in [-1,1]^{E(G)}$, and $c_i$ be the restriction of $c$ to $E(G_i)$.
Suppose that there is a strict complementarity pair $(X_i, \Omega_i)$ of ${\rm P}(G_i, \Sigma_i, c_i)$ for each $i=1,2$ such that $X_1[S,S]=X_2[S,S]$, where $S=V(G_1)\cap V(G_2)$.
Then $\Omega_1+\Omega_2$ satisfies the strict complementarity condition  
with any maximum rank solution of ${\rm P}(G,\Sigma,\omega)$ 
(where each $\Omega_i$ is regarded as a matrix of size $|V(G)|\times |V(G)|$ by appending zero columns and zero rows).

Moreover, $\Omega_1+\Omega_2$ satisfies the strict complementarity condition
even in ${\rm P}(G-F,\Sigma\setminus F,\omega)$
for any $F\subseteq E(G)$ such that $\Omega_1[i,j]+\Omega_2[i,j]=0$ for all $ij\in F$.
\end{lemma}  
\begin{proof}
Suppose that $\Omega_1[i,j]+\Omega_2[i,j]=0$ for all $ij\in F\subseteq E(G)$.
Denote $(G', \Sigma')=(G-F, \Sigma\setminus F)$ and let $c'$ be the restriction of $c$ to $F$.
Then $\Omega_1+\Omega_2$ is a dual feasible solution of  ${\rm P}(G',\Sigma', c')$.
We show that $\Omega_1+\Omega_2$ satisfies a strict complementarity condition  with any maximum rank solution of ${\rm P}(G',\Sigma', c')$.

Let $X_i=P_i^{\top} P_i$ be a Gram matrix representation of $X_i$, where each $P_i$ is row-independent.
Since $X_1[S,S]=X_2[S,S]$, each $P_i$ can be expressed as 
$P_i=\begin{array}{|c|c|} 
\hline
 \multirow{2}{*}{$\tilde{P}_i$} & 0 \\  \cline{2-2}
 & P_S  \\\hline
\end{array}$
for some row-independent matrix $P_S$ of size $d\times |S|$.
We concatenate $P_1$ and $P_2$ (by changing the row ordering of $P_2$ appropriately) such that 
\begin{equation}
\label{eq:sum}
P=\begin{array}{|c|c|c|} 
 \multicolumn{1}{r}{V(G_1)\setminus S} &  \multicolumn{1}{r}{S} &  \multicolumn{1}{r}{V(G_2)\setminus S } \\
\hline
 \multirow{2}{*}{$\tilde{P}_1$} & 0 & 0 \\  \cline{2-3}
 & P_S &   \multirow{2}{*}{$\tilde{P}_2$}  \\\cline{1-2}
 0 & 0 &\\ \hline
\end{array}.
\end{equation}
Then $P^{\top} P$ forms a feasible solution of ${\rm P}(G',\Sigma', c')$ by $E(G')\subseteq E(G_1)\cup E(G_2)$.

We show that $P^{\top} P$ and $\Omega_1+\Omega_2$ satisfy the strict complementarity condition.
Clearly $\langle P^{\top}P, \Omega_1+\Omega_2\rangle=0$, implying 
$\rank P^{\top} P\leq \dim \ker (\Omega_1+\Omega_2)$.
To show the opposite direction, take any $x\in \ker (\Omega_1+\Omega_2)\subseteq \mathbb{R}^{V(G)}$.
Note that $x$ restricted to $\mathbb{R}^{V(G_i)}$ is in $\ker \Omega_i$.
This implies that, denoting the restriction of $x$ to $\mathbb{R}^{V(G_i)}$ by $x_i$, 
$x_i$ is spanned by the rows of $P_i$.
Since $P_S$ is row independent, $x$ restricted to $\mathbb{R}^S$ is uniquely represented as a linear combination of the row vectors of $P_S$.
Hence, it follows from (\ref{eq:sum}) that the representation of $x_1$ as a linear combination of the row vectors of $P_1$ can be concatenated with that of $x_2$ as a linear combination of the rows of $P_2$ so that $x$ is represented as a  linear combination of the rows of $P$. In other words $x$ is spanned by rows of $P$,
meaning that $\rank P^{\top} P=\rank P\geq \dim \ker (\Omega_1+\Omega_2)$. 
\end{proof}

\subsection{Odd-$K_4$ minor free signed graphs}
\label{subsec:signed}
A signed graph $(G,\Sigma)$ is a pair of an undirected graph $G$ (which may contain parallel edges) and 
$\Sigma\subseteq E(G)$.
An edge in $\Sigma$ (resp.~in $E(G)\setminus \Sigma$) is called {\it odd} (resp.~even),
and a cycle (or a path) is said to be {\it odd} (resp. {\it even}) if the number of odd edges in it is odd (resp.~even).

The {\em resigning} on  $X\subseteq V$ changes $(G,\Sigma)$ with $(G,\Sigma\Delta \delta(X))$, where 
$A \Delta B:=(A\setminus B)\cup (B\setminus A)$ for any two sets $A, B$.
Two signed graphs are said to be {\em (sign) equivalent} if they can be converted to each other by a series of resigning operations.
A signed graph is called a {\it minor} of $(G,\Sigma)$ if it can be obtained by a sequence of the following three operations:
(i) the removal of an edge,
(ii) the contraction of an even edge, and
(iii) resigning.
We say that $(G,\Sigma)$ is {\em $(H, \Sigma')$ minor free}  if $(H, \Sigma')$ is not a minor of $(G,\Sigma)$.
Similarly $(H, \Sigma')$ is called an {\em odd-subdivision} of $(G,\Sigma)$ if it can be obtained from $(G, \Sigma)$ by subdividing even edges  and resigning.

A signed graph $(G,\Sigma)$ is said to be {\em bipartite} if it has no odd cycle, equivalently it is equivalent to $(G,\{\emptyset\})$.

For an undirected graph $H$,  signed graph $(H, E(H))$ is called  {\em odd-$H$}.
Also the signed graph obtained from $H$ by replacing each edge with two parallel edges with distinct signs is called {\em odd-$H^2$}.
We will frequently encounter odd-$K_4$ and odd-$K_3^2$, which are illustrated in Figure~\ref{fig:1}.

\begin{figure}
\centering
\begin{minipage}{0.45\textwidth}
\centering
\includegraphics[scale=0.6]{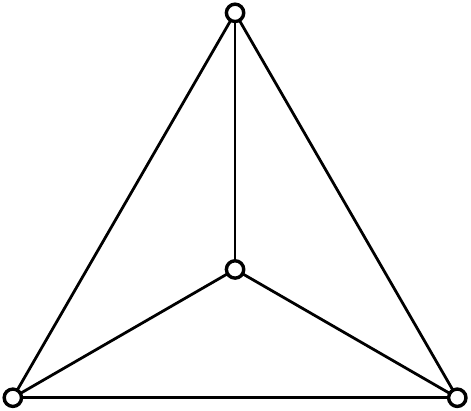}
\par(a)
\end{minipage}
\begin{minipage}{0.45\textwidth}
\centering
\includegraphics[scale=0.6]{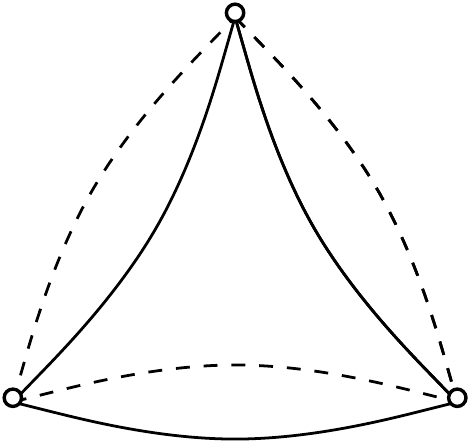}
\par (b)
\end{minipage}
\caption{Signed graphs: (a) odd-$K_4$ and (b) odd-$K_3^2$, where the even edges are dotted.}
\label{fig:1}
\end{figure}

The proof of the main theorem relies on a structural theorem of odd-$K_4$ minor free graphs. 
To see this we introduce one more notation.
Let $E_1$ and $E_2$ be nonempty subsets of $E(G)$ that partition $E(G)$,
and let $\tilde{G}_i=(V(E_i), E_i)$.
Suppose that $|V(E_1)\cap V(E_2)|=2$, $|V(E_1)|\geq 3$, $|V(E_2)|\geq 3$, and 
each $(\tilde{G}_i, \Sigma\cap E_i)$ is connected and non-bipartite.
Let $(G_i, \Sigma_i)$ be the union of $(\tilde{G}_i, \Sigma\cap E_i)$ and the odd-$K_2^2$ on the two vertices of $V(G_1)\cap V(G_2)$.
Then $(G_1, \Sigma_1)$ and $(G_2, \Sigma_2)$  are said to form a {\em strong 2-split} of $(G,\Sigma)$, and each $(G_i, \Sigma_i)$ is called the {\em parts} of the strong 2-split.
See Figure~\ref{fig:2split}.

\begin{figure}
\centering
\includegraphics[scale=0.7]{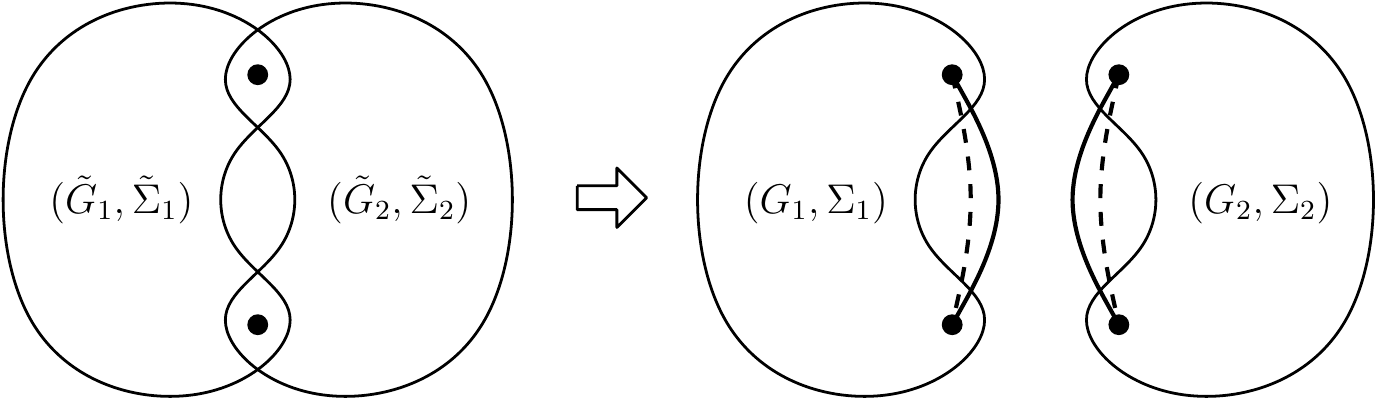}
\caption{2-split}
\label{fig:2split}
\end{figure}

The following theorem was observed by Lov{\'a}sz and Schrijver (see \cite[Theorem 3.2.5]{g}).
\begin{theorem}[Lov{\'a}sz and Schrijver]
\label{thm:structural}
Suppose that $(G,\Sigma)$ is odd-$K_4$ minor free.
Then one of the following holds.
\begin{description}
\setlength{\parskip}{0.1cm} 
 \setlength{\itemsep}{0.1cm}   
\item[(i)] $(G,\Sigma)$ has a cut vertex or a strong 2-split.
\item[(ii)] $(G,\Sigma)$ is equivalent to the odd-$K_3^2$.
\item[(iii)] $(G,\Sigma)$ is odd-$K_3^2$ minor free.
\end{description}
\end{theorem}

\subsection{Projection of the elliptope and the metric polytope}
\label{subsec:metric}

The {\em metric polytope} of an undirected graph $G$ is defined by
\begin{equation*}
{\rm MET}(G)=
\left\{x\in [0,1]^{E} \ \Bigg| \ \sum_{e\in E(C)\setminus F} x(e) - \sum_{e\in  F} x(e)
\geq 1-|F| : 
\begin{array}{cc} \text{$\forall$ cycle $C$ in $G$} \\ \text{$\forall F\subseteq C$: $|F|$ is odd} \end{array}
\right\}.
\end{equation*}
Throughout the paper we set ${\rm arccos}: [-1,1]\rightarrow [0, \pi]$ (so that ${\rm arccos}$ is bijective). Then we have the following.
\begin{theorem}[Laurent~\cite{l97}]
\label{thm:met}
For a graph $G$, ${\rm arccos}({\cal E}(G))/\pi\subseteq {\rm MET}(G)$, with equality if and only  if $G$ is $K_4$-minor free.
\end{theorem}
See \cite{dl} for more details on the metric polytope.

Suppose that we are given a signed graph $(G,\Sigma)$.
Each signing $\Sigma$ defines  a {\em sign function} $\sigma:E(G)\rightarrow \{-1, +1\}$ such that $\sigma(e)=-1$ if and only if $e\in \Sigma$.
We define a signed version of the metric polytope as follows.
\begin{equation}
\label{eq:met}
{\rm MET}(G,\Sigma)=
\left\{x\in [0,1]^E \ \Bigg| \ \sum_{e\in E(C)} \sigma(e)x(e) 
\geq 1-|E(C)\cap \Sigma| : 
\forall \text{odd cycle $C$ in $(G,\Sigma)$ }
\right\}.
\end{equation}

Note that ${\rm MET}(G)\subseteq {\rm MET}(G,\Sigma)$.
Note also that (a projection of) the metric polytope of the odd-$G^2$ coincides with ${\rm MET}(G)$.
\begin{lemma}
\label{lem:met_poly}
For a signed graph $(G,\Sigma)$, 
${\rm arccos}({\cal E}(G,\Sigma))/\pi\subseteq {\rm MET}(G,\Sigma)$.
\end{lemma}
\begin{proof}
Suppose that $c\in {\cal E}(G,\Sigma)$.
Let $X$ be a feasible solution of ${\rm P}(G,\Sigma,c)$, 
and take $p:V(G)\rightarrow \mathbb{S}^d$ such that  $X={\rm Gram}(p)$.
Then $c'\in \mathbb{R}^E$ defined by $c'({ij})=p(i)\cdot p(j)$ for $ij\in E$ should be in ${\cal E}(G)$. Thus ${\rm arccos}(c')/\pi \in {\rm MET}(G)$ by Theorem~\ref{thm:met}.
Observe also that $\sigma(e){\rm arccos}(c(e))\geq \sigma(e){\rm arccos}(c'(e))$ for any $e\in E$. 
Therefore we get $\sum_{e\in E(C)}\sigma(e){\rm arccos}(c(e))\geq \sum_{e\in E(C)} \sigma(e){\rm arccos}(c'(e))\geq (1-|E(C)\cap \Sigma|)\pi$ for any odd cycle $C$,
meaning that ${\rm arccos}(c)/\pi\in {\rm MET}(G,\Sigma)$.
\end{proof}

For simplicity, we use
\[
{\rm val}(H, x):=\sum_{e\in E(H)\setminus \Sigma} x(e)+\sum_{e\in E(H)\cap \Sigma} (1-x(e))
\]
 for each subgraph $H$ in a signed graph $(G,\Sigma)$. 
Note that $x\in {\rm MET}(G,\Sigma)$ if and only if ${\rm val}(C, x)\geq 1$ for every odd cycle $C$. 
If $x$ is clear from the context, we simply denote ${\rm val}(H)$.


\section{Characterizing ${\cal E}(G,\Sigma)$}
\label{sec:main}

Let $(G,\Sigma)$ be a signed graph  and $c\in [-1,1]^E$.
We say that an edge $e$ is {\em degenerate (with respect to $c$)} if 
\[
c(e)=\sigma(e),
\]
and that $c$ is {\em nondegenerate} if  none of the edges are degenerate, i.e., 
$c\in (-1, 1]^{\Sigma}\times  [-1,1)^{E\setminus \Sigma}$.
An odd cycle $C$ is  called {\em tight} (with respect to $c$) if 
\[{\rm val}(C,  {\rm arccos}(c)/\pi)=1.\]
An edge is  said to be {\em tight} if it is contained in some tight odd cycle,
and  called {\em strictly tight} if it is contained in a tight odd cycle of length at least three.

We are ready to state our main theorem.

\begin{theorem}
\label{thm:oddK4}
Let $(G,\Sigma)$ be an odd-$K_4$ minor free signed graph and  
$c$ be nondegenerate.
If ${\rm arccos}(c)/\pi\in {\rm MET}(G,\Sigma)$, then 
${\rm P}(G,\Sigma,c)$ is feasible and there is a dual solution which is supported on strictly tight edges and satisfies  the strict complementarity condition with a maximum rank solution of 
${\rm P}(G,\Sigma, c)$.
\end{theorem}


The proof of Theorem~\ref{thm:oddK4} will be given in Section~\ref{sec:proof}.

%
By using the following  resigning operations, the nondegeneracy assumption on $c$ can be eliminated in the first claim of Theorem~\ref{thm:oddK4}.  
Let $(G,\Sigma)$ be a signed graph and let  $S\subseteq V(G)$.
For  $x\in [0, 1]^{E(G)}$, we define the {\em resigning} $x^S\in [0,1]^{E(G)}$ of $x$ by 
$x^S(e)=1-x(e)$ if $e \in \delta(S)$ and
$x^S(e)=x(e)$ if $e\notin \delta(S)$.
Then a simple calculation shows that $x\in {\rm MET}(G,\Sigma)$ if and only if $x^S\in {\rm MET}(G, \Sigma\Delta \delta(S))$.

On the other hand, for $c\in [-1,1]^{E}$,  we define the {\em resigning} $c^S\in [0,1]^{E}$ of $c$ by 
$c^S(e)=-c(e)$ if $e \in \delta(S)$ and
$c^S(e)=c(e)$ if $e\notin \delta(S)$.
Then $c\in {\cal E}(G,\Sigma)$ if and only if $c^S\in {\cal E}(G,\Sigma\Delta \delta(S))$.
Note also that ${\rm arccos}(c^S)/\pi=x^S$ if and only if ${\rm arccos}(c)/\pi=x$.

Using resigning operations and Theorem~\ref{thm:oddK4} we can now prove the following generalization of Theorem~\ref{thm:met}.
\begin{theorem}
\label{thm:oddK4projection}
Let $(G,\Sigma)$ be a signed graph.
Then ${\rm arccos}({\cal E}(G,\Sigma))/\pi={\rm MET}(G,\Sigma)$ holds if and only if $(G,\Sigma)$ is odd-$K_4$ minor free.
\end{theorem}
\begin{proof}
Suppose that $(G,\Sigma)$ is odd-$K_4$ minor free.
By Lemma~\ref{lem:met_poly} it suffices to show that, for a given $c\in [-1,1]^E$
with ${\rm arccos}(c)/\pi \in {\rm MET}(G,\Sigma)$, 
$c\in {\cal E}(G,\Sigma)$ holds.
  
Suppose that there is an edge $ij\in E(G)\setminus \Sigma$ with $c({ij})=1$.
Let $(G',\Sigma)$ be the signed graph obtained  by contracting $ij$
and let $c'$ be the restriction of $c$ to $E(G')$.
Since ${\rm arccos}(c({ij}))=0$, 
${\rm arccos}(c')/\pi \in {\rm MET}(G',\Sigma)$ if and only if 
 ${\rm arccos}(c)/\pi \in {\rm MET}(G,\Sigma)$. 
 Also, for any solution $X$ of ${\rm P}(G,\Sigma,c)$,  
 we have $1\geq X[i,j]\geq c({ij})=1$, meaning that $X[i,j]=1$, and hence $X[i,k]=X[j,k]$ for any $k$.
 Hence  $c\in {\cal E}(G,\Sigma)$ if and only if $c'\in {\cal E}(G',\Sigma)$.
Therefore,  we may focus on the case when there is no edge $ij\in E(G)\setminus \Sigma$ with $c({ij})=1$.

Suppose that there is an edge $ij\in  \Sigma$ with $c({ij})=-1$.
Then we can consider resigning with respect to a cut $\delta(S)$ with $ij\in \delta(S)$,
which makes $ij$ even with $c^S({ij})=1$.
Hence we can again apply the same argument.
In total we can always reduce the problem to the situation when every edge is nondegenerate,
and the sufficiency follows from Theorem~\ref{thm:oddK4}.

For the necessity, the same example as the unsigned case works.
Indeed, if $(G,\Sigma)$ is odd-$K_4$ with $\Sigma=E(G)$, then 
matrix $X\in {\cal E}_4$ with $X[i,j]=-1/2 \ (i\neq j)$ is in ${\rm MET}(G,\Sigma)$ but not in ${\rm arccos}({\cal E}(G,\Sigma))/\pi$ (see \cite[Section 4]{l97}). 
This example can be extended to any signed graph having an odd-$K_4$ minor by assigning degenerate edge weight. 
\end{proof}

\section{Proof of Theorem~\ref{thm:oddK4}}
\label{sec:proof}

The proof of Theorem~\ref{thm:oddK4} consists of a series of lemmas based on the structural theorem of odd-$K_4$ minor free graphs (Theorem~\ref{thm:structural}). In Subsection~\ref{subsec:K32} we deal with the odd-$K_3^2$, and in Subsection~\ref{subsec:oddK32free} we deal with graphs having no odd-$K_3^2$ minor. In Subsection~\ref{subsec:proof} we solve the remaining case.

For simplicity of notation,  a dual solution $\omega$ in ${\rm P}(G,\Sigma,c)$ is  said to be {\em nice}  if
 it satisfies the strict complementarity condition with some (any) maximum rank solution of  ${\rm P}(G,\Sigma, c)$.

\subsection{Odd-$K_3^2$}
\label{subsec:K32}
Theorem~\ref{thm:structural} says that odd-$K_3^2$ is one of fundamental pieces when constructing odd-$K_4$ minor free graphs. In this subsection we deal with this special signed graph.
We begin with the case when $G$ is isomorphic to $K_3$ and the cycle is tight.  
\begin{lemma}
\label{K3}
Let $(G,\Sigma)$ be a signed graph and 
$c$ be nondegenerate.
Suppose that $G$ is isomorphic to $K_3$ and the cycle is odd and tight.
Then ${\rm P}(G,\Sigma, c)$ has a unique solution, whose rank is equal to two,
and there exists a nice dual solution $\omega$ supported on strictly tight edges.
\end{lemma}
\begin{proof}
%
Let $V(G)=\{1,2,3\}$.
The equation, 
\[
\sum_{e\in E(G)} \sigma(e){\rm arccos}(c(e))=(1-|\Sigma|)\pi,
\] 
of the tightness of the triangle uniquely determines (up to rotations) three points $p_1, p_2, p_3$ on the unit circle in the plane $\mathbb{R}^2$ such that $p_i\cdot p_j=c(ij)$.
Moreover, by $c\in (-1, 1]^{\Sigma}\times  [-1,1)^{E\setminus \Sigma}$, 
any two among the three are linearly independent.

Since $p_1, p_2, p_3$ lie on a plane, there is a nonzero vector $u\in \mathbb{R}^3$ such that 
$\sum_{i=1}^3 u_ip_i=0$.
If $|\Sigma|=3$ then $-p_3$ lies on the cone generated by $p_1$ and $p_2$,
meaning that all the signs of $u_i$ are the same.
If $|\Sigma|=1$ (say $\sigma(12)=+$), then $p_3$ lies on the cone generated by $p_1$ and $p_2$,
meaning that $\sigma(13)=\sigma(23)=-$ and $\sigma(12)=+$.
Moreover, since any two points  are linearly independent, 
$u_i\neq 0$ for every $i$, meaning that $u_iu_j\neq 0$ for every $i$ and $j$.
Summarizing these case analysis, we conclude that the sign of $u_iu_j$ is equal to $\sigma(ij)$ for every edge $ij$.

Let $X={\rm Gram}(p_1, p_2, p_3)$. 
Then observe that $u\in {\rm ker} X$.
Therefore $u u^{\top}$ is a feasible dual solution such that every entry is nonzero. 
Also, since $X$ has rank two, $\rank X+\rank u u^{\top}=3$.
Therefore, $(X, u u^{\top})$ satisfies the  complementarity condition.
\end{proof}

The following lemma solves the case when $(G,\Sigma)$ is odd-$K_3^2$.
\begin{lemma}
\label{K32}
Let $(G,\Sigma)$ be a signed graph equivalent to the odd-$K_3^2$ and $c$ be nondegenerate.
If ${\rm arccos}(c)/\pi\in {\rm MET}(G,\Sigma)$, then ${\rm P}(G,\Sigma, c)$ is feasible and 
there is a nice dual solution supported on strictly tight edges.
\end{lemma}
\begin{proof}
If there is a tight triangle, then the statement follows from Lemma~\ref{K3}.
Suppose that there is no tight triangle. 
We show that ${\rm P}(G,\Sigma,c)$ has a solution of rank three
(which satisfies the strict complementarity pair condition the zero matrix).

To see this we shall continuously changes the value of $c(e)$ for every non-tight edge $e$ such that 
$c(e)$ is monotone decreasing (resp.~increasing) if $e$ is odd (resp.~even) until a new tight cycle appears.
Since every pair of  vertices is linked to each other by parallel edges with distinct signs, 
$c$ keeps staying in $(-1,1]^{E\cap \Sigma}\times [-1, 1)^{E\setminus \Sigma}$.

If the new tight cycle has length two, we keep on the continuous change for the remaining non-tight edges, until either every 2-cycle is tight or a triangle becomes tight. Let $c'$ be the resulting $c$. 

If there is no tight triangle with respect to $c'$, then every two cycle is tight, 
and hence ${\rm P}(G,\Sigma,c')$ coincides with the undirected case ${\rm P}(K_3,c')$ 
(where $c'$ is naturally considered as a vector on $E(K_3)$ since $c'(e_1)=c'(e_2)$ for every pair of parallel edges $e_1$ and $e_2$).
Since ${\rm arccos}(c')/\pi\in {\rm MET}(K_3)$ and there is no tight triangle, 
${\rm arccos}(c')$  is an interior point of ${\rm MET}(K_3)$.
Thus the corresponding $X\in {\cal E}_3$ is also an interior point, meaning that $X$ is a rank-three solution $X$ of ${\rm P}(K_3,c')$.
This $X$  is also a solution of ${\rm P}(G,\Sigma,c)$. 
 
If there is a tight triangle with respect to $c'$, then 
Lemma~\ref{lem:K32acyclic} gives a rank-two solution $X$ of ${\rm P}(G,\Sigma, c')$,
which is also a solution of ${\rm P}(G,\Sigma, c)$.
More specifically, there are three points $p_1, p_2, p_3$ on a (2-dimensional) plane such that 
$p_i\cdot p_j=c'(ij)$, where $ij$ denotes the edge between $i$ and $j$ in the tight triangle.
Let $\epsilon$ be a positive number.
Since there was no tight triangle at the beginning (with respect to $c$), there must exists a pair $i$ and $j$ of vertices such that $\sigma(e)c'(e)>\sigma(e)c(e)$ for each edge $e$ between $i$ and $j$.
Without loss of generality let $i=1$ and $j=2$.
We shall slightly move $p_1$ off the plane spanned by $p_2$ and $p_3$ 
such that $|p_1\cdot p_2|$ changes at most $\epsilon$ while 
$p_2\cdot p_3$ and $p_3\cdot p_1$ are unchanged.
(For any $\epsilon\geq 0$, we can move $p$ in such a manner due to the triangle inequality on the spherical space.)
If $\epsilon$ is sufficiently small, we have
$\sigma(e) ( p_1\cdot p_2)\geq \sigma(e)c'(e)-\epsilon\geq \sigma(e)c(e)$
for any edge $e$ between $1$ and $2$.
This implies that ${\rm Gram}(p_1, p_2, p_3)$ is a feasible solution of ${\rm P}(G, \Sigma, c)$, whose rank is equal to three.
\end{proof}

For later use we extend Lemma~\ref{K3} to the following case.
\begin{lemma}
\label{cycle}
Let $(G,\Sigma)$ be a signed graph and $c$ be nondegenerate.
Suppose that $G$ is a tight odd cycle of length at least three.
Then ${\rm P}(G,\Sigma, c)$ has a unique solution $X$ with $\rank X=2$ 
and there is a nice dual solution $\omega$ such that  $\omega(e)\neq 0$ for all edges $e$.
\end{lemma}
\begin{proof}
The proof is done by induction on the number of vertices.
The base case has been done in Lemma~\ref{K3}.

Suppose that $|V(G)|\geq 4$.
Take any vertex $v$ and let $i, j$ be the neighbors of $v$.
Without loss of generality (by resigning), assume that $vi$ and $vj$ are both even.
We first remark
\begin{equation}
\label{eq:cycle1}
0<{\rm arccos}(c(iv))+{\rm arccos}(c(vj))<\pi.
\end{equation}
The first inequality follows from the fact that $c$ is nondegenerate and hence $c(iv)\neq 1$ and $c(vj)\neq 1$.
The second inequality follows since $G$ is a tight odd cycle and hence
$1={\rm val}(G, {\rm arccos}(c)/\pi)=\sum_{e\in E(G)\setminus \Sigma}{\rm arccos}(c)/\pi
+\sum_{e\in \Sigma}(1-{\rm arccos}(c)/\pi)
> {\rm arccos}(c(iv))/\pi+{\rm arccos}(c(vj))/\pi$
(where the last inequality follows from the fact that $\Sigma\neq \emptyset$ and $c$ is nondegenerate).

Let $(G', \Sigma)=(G-v+ij, \Sigma)$ be the signed graph obtained from $(G,\Sigma)$ by removing $v$ and inserting a new even edge between $i$ and $j$,
and define $c'\in [-1,1]^{E(G')}$ by $c'(e)=c(e)$ for $e\in E(G')\setminus \{ij\}$ and 
\[{\rm arccos}(c'(ij))={\rm arccos}(c(iv))+{\rm arccos}(c(vj))\] for the new edge $ij$. 
By (\ref{eq:cycle1}), $c'$ is well-defined. (Recall that ${\rm arccos}$ is assumed to be a bijection between $[-1,1]$ and $[0,\pi]$.)
Note that $(G', \Sigma)$ is odd and tight with respect to $c'$. 
Also, by (\ref{eq:cycle1}), $c'$ is nondegenerate.
Therefore, by induction, ${\rm P}(G', \Sigma', c')$ has a  pair $(X', \Omega')$ satisfying the strict complementarity condition such that $X'$ has rank two and $\omega'(e)\neq 0$ for every edge $e$.

Let $K_3$ be the complete graph on $\{v, i, j\}$, and define $\tilde{c}\in \mathbb{R}^{E(K_3)}$ by
$\tilde{c}(ij)=c'(ij)$ and $\tilde{c}(e)=c(e)$ for $e\in E(K_3)\setminus \{ij\}$.
Then signed graph $(K_3, \{ij\})$ is odd and tight with respect to $\tilde{c}$.
Also, by (\ref{eq:cycle1}), $\tilde{c}$ is nondegenerate.
Hence, by Lemma~\ref{K3}, ${\rm P}(K_3, \{ij\}, \tilde{c})$ has a  pair $(\tilde{X}, \tilde{\Omega})$ satisfying the strict complementarity condition such that $\tilde{X}$ has rank two and $\tilde{\omega}(e)\neq 0$ for every edge $e$.

By the complementarity condition, we have $X'[i,j]=c(ij)=\tilde{X}[i,j]$, meaning that we can glue $X'$ and $\tilde{X}$ to form a matrix $X$ of size $|V(G)|\times |V(G)|$ and rank two.
Note also that, since $ij$ is even in $(G', \Sigma')$ and $ij$ is odd in $(K_3, \{ij\})$, we may assume (by scaling by a positive number) that $\omega'(ij)+\tilde{\omega}(ij)=0$.
Now, by regarding $\Omega'$ and $\tilde{\Omega}$ as matrices of size $|V(G)|\times |V(G)|$ (by appending zero rows and zero columns), let $\Omega=\Omega'+\tilde{\Omega}$. 
Since $\omega'(ij)+\tilde{\omega}(ij)=0$, $\Omega$ is dual feasible in ${\rm P}(G, \Sigma, c)$.
By Lemma~\ref{lem:glueing}, $(X, \Omega)$ satisfies the  strict complementarity condition.

Since $\omega(ij)\neq 0$ for every edge $ij$,  $Y[i,j]=c(ij)$ for any feasible solution $Y$ of ${\rm P}(G,\Sigma, c)$. 
Since $(G,\Sigma)$ is a tight cycle, the tight equation uniquely determines the other entries of $Y$,
which implies the uniqueness of the solution.
\end{proof}

\subsection{Odd-$K_3^2$ minor free graphs}
\label{subsec:oddK32free}

Let $(G,\Sigma)$ be a signed graph and ${\cal C}$ be a set of odd cycles.
We denote $E({\cal C}):=\bigcup_{i=1}^k E(C_i)$, $V({\cal C}):=\bigcup_{i=1}^k V(C_i)$,
and $G[{\cal C}]:=(V({\cal C}), E({\cal C}))$.
Also, based on ${\cal C}$ we  define a hypergraph ${\cal H}({\cal C})$ on $V(G)$ as follows.
We first construct a hypergraph ${\cal H}'$ on $V(G)$ by regarding each cycle in ${\cal C}$ as a  hyperedge. 
Then we construct hypergraph ${\cal H}({\cal C})$ from ${\cal H}'$ by greedily margining two hyperedges $e_1$ and $e_2$ with $|e_1\cap e_2|\geq 2$ into a single hyperedge. (The resulting hypergraph is uniquely determined.)
Hence in ${\cal H}({\cal C})$ we have $|e\cap e'|\leq 1$ for any two hyperedges.

Our goal is to show Lemma~\ref{lem:K32free} which proves Theorem~\ref{thm:oddK4} for odd-$K_3^2$ minor free graphs. Since the proof is a bit involved, we split it into three parts. 
In Section~\ref{subsec:acyclic}, we prove acyclicity of ${\cal H}({\cal C})$, which is a key property of expanding 
a low-rank solution to a maximum rank solution.
In Section~\ref{subsub:path}, we give a tool for inductively constructing a completion.
In Section~\ref{subsub:K32}, we prove Lemma~\ref{lem:K32free} by using the tools given in Sections~\ref{subsub:path} and~\ref{subsub:K32}.

\subsubsection{Acyclicity of ${\cal H}({\cal C})$}
\label{subsec:acyclic}
We say that a hypergraph is {\em acyclic} if the underlying edge-vertex incidence bipartite graph  has no cycle.
The following is an important property of odd-$K_3^2$ minor free graphs.
\begin{lemma}
\label{lem:K32acyclic}
Let $(G,\Sigma)$ be a signed graph and ${\cal C}$ be a set of odd cycles.
Suppose that $(G,\Sigma)$ is odd-$K_4$ minor free and odd-$K_3^2$ minor free.
Then ${\cal H}({\cal C})$ is acyclic.
\end{lemma}
In order to prove this lemma we have to introduce a special signed graph.
The odd-$K_4^-$ is a signed graph obtained from the odd-$K_4$ by removing an edge.
Let $u,v$ be a pair of distinct vertices in a signed graph $(G,\Sigma)$. 
We say that $(G,\Sigma)$ has an odd-$K_4^-$-minor {\em rooted at $\{u,v\}$} if 
it contains an odd-subdivision of the odd-$K_4^-$ in which $u$ and $v$ have degree two and are not contained in an odd cycle.
(In other words, there are five internally vertex disjoint paths as in Figure~\ref{fig:ladder}(a) such that the cycle through $a, b, u$ and the cycle through $a, b, v$ are both odd.)

\begin{figure}
\centering
\begin{minipage}{0.4\textwidth}
\centering
\includegraphics[scale=0.6]{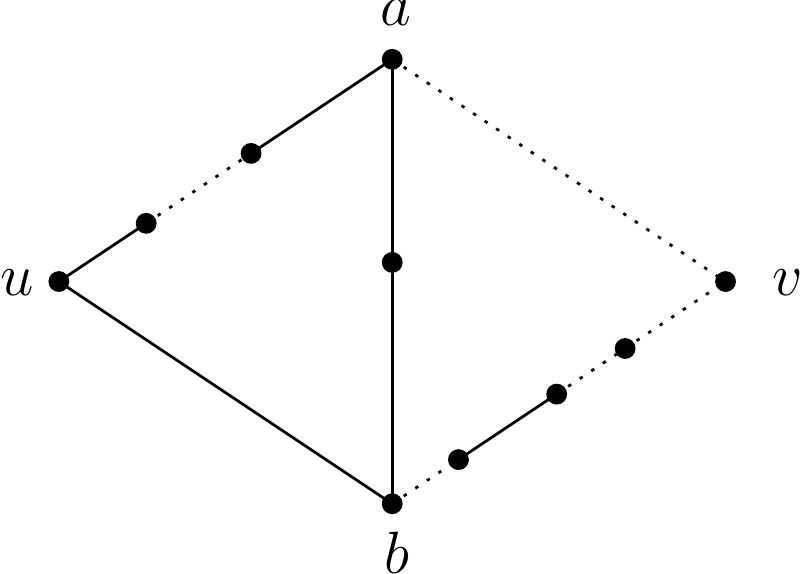}
\par(a)
\end{minipage}
\begin{minipage}{0.55\textwidth}
\centering
\includegraphics[scale=0.65]{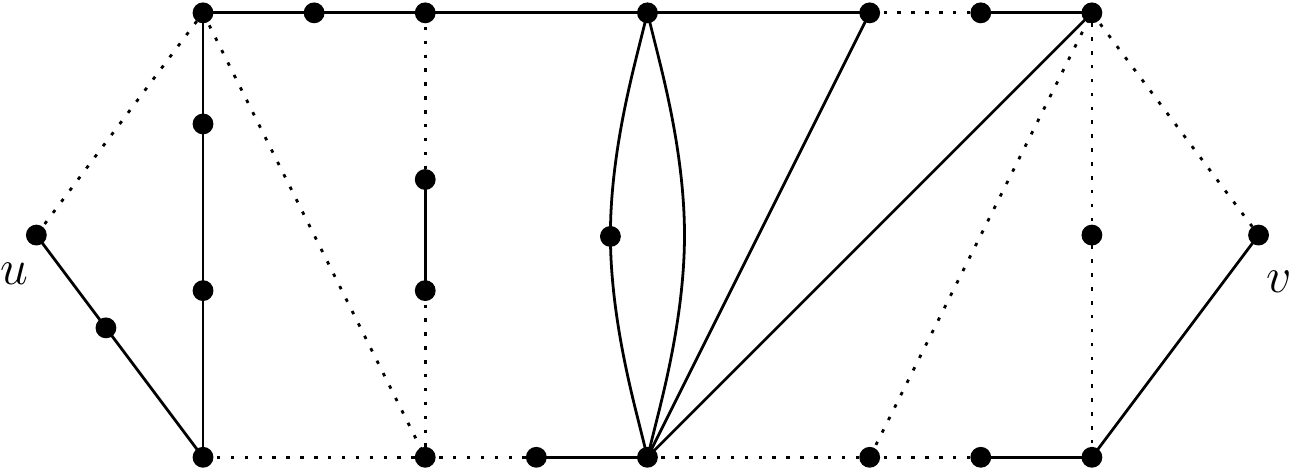}
\par (b)
\end{minipage}
\caption{(a) An odd-subdivision of the odd-$K_4^-$ rooted at $\{u,v\}$ and (b) an odd ladder rooted at $\{u,v\}$, where dotted edges are even and bold edges are odd.}
\label{fig:ladder}
\end{figure}

\begin{lemma}
\label{lem:K32acyclic2}
Let $(G,\Sigma)$ be a signed graph and ${\cal C}$ be a set of odd cycles.
Let $e$ be  a hyperedge in ${\cal H}({\cal C})$ and let $(G',\Sigma')$ be the subgraph of  $(G,\Sigma)$ induced by (the vertices of) $e$. Then  for each pair of distinct vertices $u, v\in V(G')$, $(G',\Sigma')$ has an odd cycle passing through $u$ and $v$ or an  odd-$K_4^-$ minor rooted at $\{u,v\}$.
\end{lemma}
\begin{proof}
We say that a signed graph satisfies property (A) if  for each pair $u, v$ of vertices it has an odd cycle through $u, v$ or an odd-$K_4^-$ minor rooted at $\{u,v\}$.
From the construction of ${\cal H}({\cal C})$ (based on ${\cal H}'$), it suffices to show the following:
Given two 2-connected signed graphs $(G_1, \Sigma_1)$ and $(G_2, \Sigma_2)$ with $|V(G_1)\cap V(G_2)|\geq 2$ satisfying property (A),  their union satisfies the property (A).

Now take any $u, v$ of $V(G_1)\cup V(G_2)$ with $u\in V(G_1)\setminus V(G_2)$ and $v\in V(G_2)\setminus V(G_1)$.
Since $(G_1, \Sigma_1)$ satisfies (A), $(G_1, \Sigma_1)$ contains an odd cycle $C$ that passes though $u$.
Since $G_1\cup G_2$ is 2-connected, 
there are two paths $P_1$ and $P_2$ between $v$ and $V(C)$ such that 
$|V(P_i)\cap V(C)|=1$ for $i=1,2$ and 
$V(P_1)\cap V(P_2)=\{v\}$.
Then $P_1\cup P_2\cup C$ contains an odd cycle through $u, v$ or odd-$K_4^-$ minor rooted at $\{u,v\}$.
Indeed, denoting $V(P_i)\cap V(C)=\{a_i\}$ for $i=1,2$,
let $Q_i$ be the path between $u$ and $a_i$ in $C$ such that $V(Q_1)\cap V(Q_2)=\{u\}$.
Then $P_i\cup Q_i\ (i=1,2)$ are internally vertex-disjoint paths between $u$ and $v$. 
If $P_1\cup Q_1$ and $P_2\cup Q_2$ are both even or both odd, then 
$P_1\cup P_2\cup C$ forms an odd-$K_4^-$ minor rooted at $\{u,v\}$.
\end{proof}

\begin{proof}[Proof of Lemma~\ref{lem:K32acyclic}]
Suppose that ${\cal H}({\cal C})$ is not acyclic.
Take a shortest cycle $v_1e_1\dots v_k e_k v_1$ in the underlying vertex-edge incidence bipartite graph of ${\cal H}({\cal C})$.
Since $|e\cap e'|\leq 1$, $k\geq 3$ holds. 
By Lemma~\ref{lem:K32acyclic2} 
the subgraph $(G_i, \Sigma_i)$ of $(G,\Sigma)$ induced by $e_i$ contains an odd cycle through $v_i, v_{i+1}$ or odd-$K_4^-$ minor rooted at $\{v_i, v_{i+1}\}$.
If $G_1$  contains an odd-$K_4^-$ minor rooted at $\{v_1, v_2\}$, then 
$G_2\cup \dots \cup G_k$ contains an odd path and an even path between $v_1$ and $v_2$, which are internally disjoint from $G_1$. Hence $(G,\Sigma)$ contains an odd-$K_4$ minor, contradicting the assumption.
Thus $G_i$ contains an odd cycle through $v_i, v_{i+1}$ for every $1\leq i\leq k$.
However, since $k\geq 3$, this would imply that $G_1\cup \dots \cup G_k$ contains an odd-$K_3^2$ minor, which is a contradiction.
\end{proof}

\subsubsection{Path reduction}
\label{subsub:path}

In order to explicitly construct solutions for odd-$K_3^2$ minor free signed graphs, 
we generalize the proof of Lemma~\ref{cycle} based on the following reduction technique.
Suppose that $(G, \Sigma)$ is a signed graph and $c\in [-1, 1]^E$,
and let $P$ be a path  with $P\cap \Sigma= \emptyset$ contained in a tight odd cycle $C$.
Also let $u, v$ be the two endvertices of $P$.
Note that
\begin{equation}
\label{eq:path_reduction}
0\leq \sum_{e\in E(P)}{\rm arccos}(c(e))\leq \pi
\end{equation}
since $1={\rm val}(C, {\rm arccos}(c)/\pi)\geq \sum_{e\in E(P)}{\rm arccos}(c(e))/\pi$.
We say that $(G', \Sigma', c')$ is the {\em path-reduction} of $(G,\Sigma,c)$ through $P$, if 
$(G', \Sigma')$ is obtained from $(G,\Sigma)$ by removing all the internal vertices  in $V(P)\setminus \{u,v\}$ and inserting a new even edge $uv$ 
and $c'\in [-1,1]^{E(G')}$ is obtained by setting ${\rm arccos}(c'(uv))=\sum_{e\in E(P)} {\rm arccos}(c(e))$ for the new edge $uv$ and $c'(e')=c(e')$ for the remaining edges $e'$ of $G'$.  
By (\ref{eq:path_reduction}) $c'$ is well-defined.
%
The path-reduction through a general path $P$ (where $P\cap \Sigma$ may not be empty) is accordingly defined by using resining operations.

Primal solutions of the completion problems will be constructed inductively by using path-reductions,
and our next goal is to prove Lemma~\ref{lem:ear_reduction} which states that there always exists a path through which the path-reduction preserves a certain property of the auxiliary hypergraph ${\cal H}({\cal C})$.
In order to prove this, we need a statement similar to  Lemma \ref{lem:K32acyclic2} for a set of tight odd cycles.
Because of the restriction to tight cycles, we have to further extend the concept of odd-$K_4^-$ rooted minor as follows.

A {\em wheel}  is an undirected graph with $n$ vertices formed by connecting a vertex to all vertices of the cycle of length $n-1$. 
The vertex adjacent to all other vertices is called the {\em center vertex}.
Note that a graph  is allowed to contain parallel edges, and hence a wheel may contain parallel edges incident to the center vertex. 
The graph obtained from a wheel by removing one edge not incident to the center is called a {\em fan}.

Given a signed graph $(G,\Sigma)$, a sequence ${\cal C}=(C_1, \dots, C_k)$ of distinct cycles  is said to be  a {\em tight odd ladder} if 
\begin{itemize}
\item each $C_i$ is a tight odd cycle,
\item for each $i$ with $1\leq i\leq k-1$, $C_i\cap C_{i+1}$ forms a path of length at least one, and
\item for each  $i, j$ with  $1\leq i\leq  j\leq k$,  if $V(C_i)\cap V(C_j)\neq \emptyset$, then $\bigcup_{\ell=i}^j C_{\ell}$ forms a subdivision of a fan.
\end{itemize}
See Figure~\ref{fig:ladder}(b) for an example.
A tight odd ladder ${\cal C}=(C_1,\dots, C_k)$ is said to be {\em rooted at $\{u,v\}$} if 
$u\in V(C_1)\setminus V(\bigcup_{i=2}^k C_i)$ and $v\in V(C_k)\setminus V(\bigcup_{i=1}^{k-1} C_i)$.
Similarly, for an edge $e$ and a vertex $v$,  ${\cal C}$ is said to be {\em rooted at $\{e,v\}$} if 
$e\in E(C_1)\setminus \bigcup_{i=2}^k E(C_i)$ and $v\in V(C_k)\setminus \bigcup_{i=1}^{k-1} V(C_i)$.
To prove Lemma \ref{lem:ear_reduction}, we need the following two technical lemmas.

\begin{lemma}
\label{lem:two_cycles}
Let $(G,\Sigma)$ be a signed graph, $c\in [-1,1]^E$ with ${\rm arccos}(c)/\pi\in {\rm MET}(G)$,
and  $C_1$ and $C_2$ be two distinct tight odd cycles with $|V(C_1)\cap V(C_2)|\geq 2$.
Also let $u$ and $v$ be two distinct vertices in $V(C_1)\cap V(C_2)$,
and suppose that $C_i$ is decomposed into two paths $P_i^1$ and $P_i^2$ between $u$ and $v$ for each $i=1, 2$. 
If $P_1^1$ is internally disjoint from $C_2$, then either $P_1^1\cup P_2^1$ or $P_1^1\cup P_2^2$ is  a tight odd cycle.
\end{lemma}
\begin{proof}
Since $P_1^1$ is internally disjoint from $C_2$, $P_1^1\cup P_2^1$ and $P_1^1\cup P_2^2$ are cycles,
and either one of them is odd. 
Without loss of generality, assume that $P_1^1\cup P_2^1$ is odd.
Consider the graph obtained from $P_1^2\cup P_2^1$ by regarding each edge in $E(P_1^2)\cap E(P_2^1)$ as parallel edges. 
Then this graph is Eulerian, and it can be decomposed into edge-disjoint cycles $E_1, \dots, E_k$ with $k\geq 1$.
Moreover, there is at least one odd cycle among $E_i$, since otherwise $P_1^1\cup P_2^1$ would be even.
For each $i$, we have ${\rm val}(E_i)\geq 0$ if $E_i$ is even and otherwise ${\rm val}(E_i)\geq 1$.
Therefore 
$2={\rm val}(C_1)+{\rm val}(C_2)={\rm val}(P_1^1\cup P_2^1)+\sum_{i=1}^k {\rm val}(E_i)\geq k'+1$,
where $k'$ is the number of odd cycles among $E_i$, which is at least one.
Thus  ${\rm val}(P_1^1\cup P_2^1)=1$ holds, implying that $P_1^1\cup P_2^1$ is tight. 
\end{proof}

\begin{lemma}
\label{lem:ladder}
Let $(G,\Sigma)$ be a signed graph which is odd-$K_4$ minor and odd $K_3^2$ minor free, 
$c\in [-1,1]^E$ with ${\rm arccos}(c)/\pi\in {\rm MET}(G)$, and $u, v$ be distinct vertices.
Also, let ${\cal C}=(C_1,\dots, C_k)$ be a tight odd ladder  rooted at $\{u,v\}$, 
$C$ be a tight odd cycle passing through $u$, and $e^*$ be an edge in $C$ incident to  $u$. 
Suppose that $|V(C)\cap V({\cal C})|\geq 2$, then there is a tight odd ladder ${\cal D}=(D_1,\dots, D_k)$ rooted at $\{e^*,v\}$ with $E({\cal D})\subseteq E({\cal C})\cup E(C)$. 
\end{lemma}
\begin{proof}
The statement is clear if  $e^*\in E({\cal C})$.
Hence we assume $e^*\notin E({\cal C})$.
The proof is done by induction on $k$.
If $k=1$, then the statement follows from Lemma~\ref{lem:two_cycles} by $|V(C_1)\cap V(C)|\geq 2$.
Thus we assume $k\geq 2$.

Consider tracing $C$ from $u$ in the direction to $e^*$, and let $x$ be the first vertex in $V({\cal C})$ encountered during the tracing. 
Let $P$ be the path between $u$ and $x$ in $C$ that contains $e^*$, which is internally disjoint from $G[{\cal C}]$.

Suppose that $x\in V(C_1)$.
Then $C_1$ is decomposed into two paths $P_1$ and $P_2$ between $u$ and $x$,
and either $P\cup P_1$ or $P\cup P_2$ is a tight odd cycle by Lemma~\ref{lem:two_cycles}.
Without loss of generality, assume $P\cup P_1$ is a tight odd cycle.
If $E(P\cup P_1)\cap E(C_2)\neq \emptyset$ then $\{P\cup P_1, C_2, \dots, C_k\}$ is a tight odd ladder,
and otherwise  $\{P\cup P_1, C_1, \dots, C_k\}$ is a tight odd ladder. 
The resulting tight odd ladder satisfies the condition of the statement.

 Hence we may  suppose that $x\notin V(C_1)$.
 Let $s$ (resp., $s'$) be the smallest (resp., largest) index $i$ such that $x\in V(C_i)$.
 We first solve the following special case.
 \begin{claim}
 The statement holds if $|V(C)\cap V(C_s\cup \dots \cup C_k)|\geq 2$.
 \end{claim}
 \begin{proof}
 Suppose first that $|V(C)\cap V(C_{s'}\cup \dots \cup C_k)|\geq 2$.
 Then ${\cal C}_x:=(C_{s'}, \dots, C_k)$ is a tight odd ladder rooted at $\{x, v\}$, and $C$ is a tight odd cycle passing through $x$ with $|V(C)\cap V({\cal C}_x)|\geq 2$.
Let $f$ be the edge  in $P$ incident to  $x$.
Since $s'\geq s\geq 2$ by $x\notin V(C_1)$, we can apply induction to get a tight odd ladder ${\cal C}_y$ rooted at $\{f, v\}$ with $E({\cal C}_y)\subseteq E({\cal C}_x)\cup E(C)$.
Since $P$ is internally disjoint from $G[{\cal C}_x]$ 
and only the first cycle of ${\cal C}_y$ contains $f$,
 the first cycle in ${\cal C}_y$ contains $P$ but the remaining cycles do not.
 In particular, $e^*$ is contained only in the first cycle of ${\cal C}_y$, meaning hat ${\cal C}_y$ is also a tight odd ladder rooted at $\{e^*, v\}$.

Hence we may suppose $V(C)\cap V(C_{s'}\cup \dots \cup C_k)=\{x\}$.
Take the maximum index $t$ such that $|V(C)\cap V(C_{t})|\geq 2$.
By our assumption, we have $s\leq t<s'$.
By Lemma~\ref{lem:two_cycles}, $C\cup C_t$ contains a tight odd cycle $C'$ that contains $e^*$.
If $E(C')\cap E(C_{t+1})\neq \emptyset$ then $\{C', C_{t+1}, \dots, C_k\}$ is a tight odd ladder,
and otherwise  $\{C', C_t, \dots, C_k\}$ is a tight odd ladder with the required property.
\end{proof}

Thus we assume $|V(C)\cap V(C_s\cup \dots \cup C_k)|=1$. In particular $|V(C)\cap V(C_s)|=1$.
\begin{claim}
$C_1\cup \dots \cup C_s$ forms a fan
\end{claim}
\begin{proof}
Suppose that $C_1\cup \dots \cup C_s$ does not form a fan.
Then let $t$ be the largest index such that $C_1\cup \dots \cup C_t$ forms a fan.
Then $C_1\cup \dots \cup C_t$ contains an odd-subdivision $H$ of the odd-$K_4^-$ rooted at $\{w, u\}$ for some $w\in C_t$ with $C_1\subseteq H$. 
Since $C_1\cup \dots \cup C_s$ does not form a fan, $C_s$ and $C_1$ are vertex-disjoint.
Hence $C_{t}\cup C_{t+1}\cup \dots \cup C_s\cup P$ contains an even path $P_e$ and an odd path $P_o$ from $u$ to $w$ avoiding $V(C_1)$.  
This implies that $H\cup P_e$ or $H\cup P_o$ forms an odd-subdivision of the odd-$K_4$, contradicting the odd-$K_4$ freeness.
\end{proof}

Hence   $C_1\cup \dots \cup C_s$ forms a fan.
Let $c$ be the center vertex of the fan.
Without loss of generality, we can suppose (by resigning) that there is exactly one odd edge in each $C_i\ (1\leq i\leq s)$, and the odd edge in $C_i$ is incident to $c$.
Note that $c\neq x$ by $x\notin V(C_1)$.

Consider tracing $C$ from $u$ in the direction to $e^*$, and let $y$ be the first vertex in $V(C_1)$ encountered during the tracing. ($y$ should be encountered after $x$.)
Let $Q$ be the path from $x$ to $y$ in $C$ passed during the tracing.
See Figure~\ref{fig:cases}.
Since $|V(C)\cap V(C_s)|=1$, $V(Q)\cap V(C_s)=\{x\}$.
We have two cases.

Suppose that $|V(C_s)\cap V(C_1)|=1$. (See Figure~\ref{fig:cases}(a).)
Then the sign of $P$ should be equal to that of $Q$ since otherwise  $C_1\cup C_s\cup P\cup Q$ forms a minor of the odd-$K_3^2$. As $C$ is odd, this implies $y\neq u$.
Let $R$ be the path between $y$ and $c$ in $C_1$ that avoids $u$, and 
$S$ be the path between $u$ and $c$ in $C_1$ that avoids $y$.
Take the path $T$ between $c$ and $x$ in $C_s$ such that $Q\cup R \cup T$ forms an odd cycle.
\begin{claim}
\label{claim:ladder}
$P\cup S\cup (C_s\setminus T)$ is a tight odd cycle.
\end{claim}
\begin{proof}
Note that $P\cup S\cup (C_s\setminus T)$ forms a cycle.
To see the tightness, recall that every odd edge of $C_i$ is incident to $c$. Hence  $C_1\setminus (R\cup S)$ is even. Hence $R$ and $S$ have different signs.
Since the sign of $P$ is the same as that of $Q$, the sign of $P\cup S\cup (C_s\setminus T)$ is the same as that of $Q\cup R \cup T$, which is odd.
Now $(C_1\setminus (R\cup S)\cup (C\setminus (P\cup Q))$ is Eulerian, which can be decomposed into cycles $E_1, \dots, E_{\ell}$. 
Since $P$ and $Q$ have the same sign,  $C\setminus (P\cup Q)$ is an odd path.
Hence there is at least one odd cycle among  $E_1, \dots, E_{\ell}$.
Therefore, 
$3={\rm val}(C)+{\rm val}(C_1)+{\rm val}(C_k)={\rm val}(P\cup S\cup (C_s\setminus T))+{\rm val}(Q\cup R \cup T)+\sum_{i=1}^{\ell} {\rm val}(E_i)\geq 3$, implying that ${\rm val}(P\cup S\cup (C_s\setminus T))=1$.
In other words, $P\cup S\cup (C_s\setminus T)$ forms a tight odd cycle. 
\end{proof}

Hence, depending on whether $E(P\cup S\cup (C_s\setminus T))\cap E(C_{s+1})=\emptyset$ or not, 
either $\{P\cup S\cup (C_s\setminus T), C_s, C_{s+1}, \dots, C_k\}$ or $\{P\cup S\cup (C_s\setminus T), C_{s+1}, \dots, C_k\}$ forms a tight odd ladder rooted at $\{e^*, v\}$.

Suppose finally that $|V(C_s)\cap V(C_1)|>1$. (See Figure~\ref{fig:cases}(b), where the vertical edge at the center may be two parallel paths.) 
Observe first that $P$ and $Q$ must have the same sign, since otherwise $C_1\cup C_k\cup P$ or $C_1\cup C_k\cup Q$ contains an odd-subdivision of the odd-$K_4$.
Based on this fact, the same proof as Claim~\ref{claim:ladder} implies that  there is a path $U$ from $u$ to $x$ in $C_1\cup C_k$ such that $P\cup U$ forms a tight odd cycle, and $\{P\cup U, C_s, C_{s+1}, \dots, C_k\}$ or $\{P\cup U, C_{s+1}, \dots, C_k\}$ forms a tight odd ladder rooted at $\{e^*, v\}$.
This completes the proof of the lemma.
\end{proof}

\begin{figure}
\centering
\begin{minipage}{0.45\textwidth}
\centering
\includegraphics[scale=0.7]{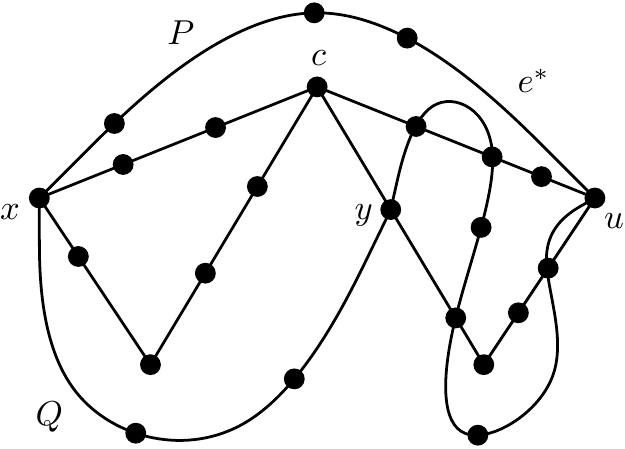}
\par
(a)
\end{minipage}
\begin{minipage}{0.45\textwidth}
\centering
\includegraphics[scale=0.7]{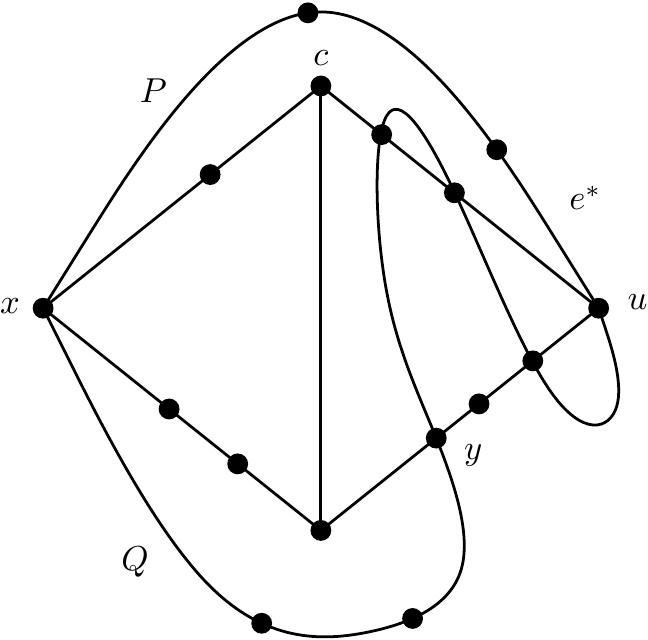}
\par
(b)
\end{minipage}
\caption{Proof of Lemma~\ref{lem:ladder}.}
\label{fig:cases}
\end{figure}

\begin{lemma}
\label{lem:ear_reduction}
Let $(G,\Sigma)$ be a signed graph which is odd-$K_3^2$ minor and odd-$K_4$ minor free with $|V(G)|\geq 3$, 
and let $c\in [-1,1]^E$ with  ${\rm arccos}(c)/\pi\in {\rm MET}(G,\Sigma)$.
Also let ${\cal C}$ be the family of tight odd cycles with respect to $c$.
Suppose that ${\cal H}({\cal C})$ consists of just one hyperedge spanning all the vertices. Then the following hold.
\begin{itemize}
\item $G$  has a path $P$ satisfying the following properties:
\begin{itemize}
\item it is contained in a tight odd cycle and its length is at least two,
\item ${\cal H}({\cal C}')$ consists of just one hyperedge spanning all the vertices of $G'$, where ${\cal C}'$ is the family of tight odd cycles in the path-reduction $(G', \Sigma', c')$ through $P$.
\end{itemize}
\item For any $u, v\in V(G)$, there is a tight odd ladder   rooted at $\{u,v\}$.
\end{itemize}
\end{lemma}
\begin{proof}
We prove the two claims  by induction on $|V(G)|$ simultaneously. 
In the base (i.e., $|V(G)|=3$), since $(G,\Sigma)$ is odd-$K_3^2$ minor free and ${\cal H}({\cal C})$ consists of just one hyperedge spanning all the vertices, 
$G$ contains a tight odd triangle.  
Thus the second claim holds. Also any two edges in the tight odd triangle  forms a path satisfying the property of the first claim. 

Thus we assume $|V(G)|\geq 4$.
To see the first claim, we consider the following process for constructing a family ${\cal D}$ of tight odd cycles:
\begin{itemize}
\item First, take any tight odd cycle $C_1$ in $G$ and set ${\cal D}=\{C_1\}$; 
\item Then, greedily add a new tight odd cycle $C'$ of $G$ to ${\cal D}$ if $C'\notin {\cal D}$ and  $|V(C')\cap V({\cal D})|\geq 2$. Repeat until every tight odd cycle $C'\notin {\cal D}$ satisfies $|V(C')\cap V({\cal D})|\leq 1$.
\end{itemize}
If the first tight odd cycle $C_1$ in the process spans $V(G)$, then any two consecutive edges forms a path with the desired property. Thus we assume that $V(C_1)\neq V(G)$ at the beginning of the process.
\begin{claim}
\label{claim:ear1}
At the end of the process, $V({\cal D})=V(G)$.
\end{claim}
\begin{proof}
Suppose that  ${\cal D}$ does not span $V(G)$ at the end.
Since ${\cal H}({\cal C})$ consists of just one hyperedge spanning all the vertices (by the lemma assumption), $(G,\Sigma,c)$ has a family ${\cal D}'$ of tight odd cycles disjoint from ${\cal D}$  with  $|V({\cal D'})\cap V({\cal D})|\geq 2$. Due to the maximality of ${\cal D}$, we have   $E({\cal D})\cap E({\cal D}')=\emptyset$.
Also $V({\cal D'})\neq V(G)$.
Let $X=V({\cal D})\cap V({\cal D}')$.

Suppose that there are two distinct vertices $u, v\in X$ which are contained in a tight odd cycle $C$ in $G[{\cal D}]$. By induction $G[{\cal D}']$ contains a tight odd ladder $(D_1, \dots, D_k)$  rooted at $\{u,v\}$. 
Since $\{u,v\}\subseteq V(C)\cap V({\cal D}')$, Lemma~\ref{lem:ladder} implies that 
there is a tight odd cycle $D_1$ intersecting both $E(C)$ and $E({\cal D}')$ (by taking the first cycle in the sequence of cycles obtained by Lemma~\ref{lem:ladder}).
Since $E({\cal D})\cap E({\cal D}')=\emptyset$, this further implies that $D_1'\notin {\cal D}$. 
However, since  $|V(D_1')\cap V({\cal D})|\geq 2$, this contradicts the maximality of ${\cal D}$.

Thus $G[{\cal D}]$ contains no tight odd cycle intersecting more than one vertex in $X$.
We take two distinct vertices $s$ and $t$ in $X$ such that the sequence  of a tight odd ladder  rooted at $\{s, t\}$ is as short as possible in $G[{\cal D}]$.
Let  ${\cal C}_{st}=(C_1, \dots, C_{\ell})$ be a tight odd ladder at $\{s,t\}$ with smallest $\ell$.
Since $G[{\cal D}]$ contains no tight odd cycle intersecting more than one vertex in $X$, 
the minimality of $|{\cal C}_{st}|$ implies that $V({\cal C}_{st})\cap X=\{s,t\}$.
Since $G[{\cal D}']$ contains a tight odd radder rooted at $\{s,t\}$, 
$G[{\cal D}']$ has an odd path $P_1$ and an even path $P_2$ between $s$ and $t$ (which may  not be internally disjoint).
By $V({\cal C}_{st})\cap X=\{s,t\}$, $P_i$ is internally disjoint from $G[{\cal C}_{st}]$ for each $i=1,2$.
Thus, by $\ell\geq 2$, $P_1\cup G[{\cal C}_{st}]$ or $P_2\cup G[{\cal C}_{st}]$ contains an odd-subdivision of the odd-$K_4$, contradicting the assumption.
\end{proof}

By Claim \ref{claim:ear1}, $V({\cal D})=V(G)$ at the end.
Since $V(C_1)\neq V(G)$, 
this means that, during the above greedy process of constructing ${\cal D}$, there was a moment at which we find a tight odd cycle $C'$ with $|V(C')\cap V({\cal D})|\geq 2$, 
$V({\cal D})\neq V(G)$, and $V(C')\cup V({\cal D})=V(G)$.
Let $P$ be a maximal path in $C'$ internally disjoint from $G[{\cal D}]$. 
Then $P$ is a path claimed in the statement.
This completes the proof of the first claim.

\medskip

To see the second claim, we take a path $P^*$ proved in the first claim.
Let $e^*$ be the new edge in the path-reduction $(G', \Sigma', c')$ through $P^*$, and let $s, t$ be the endvertices of $e^*$.
(Hence $s$ and $t$ are the endvertices of $P^*$.)
 By induction, for any $u, v$ in $G'$, there is a tight odd ladder rooted at $\{u,v\}$ in the path-reduction $(G', \Sigma', c')$.
For any tight odd cycle $C$ passing through $e^*$, $(C-e^*)\cup P$ is odd and tight due to the definition of $c'$.
Hence a tight odd ladder  rooted at $\{u,v\}$ in $(G', \Sigma', c')$ can be modified to be a tight odd ladder rooted at $\{u,v\}$ in $(G, \Sigma, c)$.
 Thus, what remains to show  the second claim is that 
 $(G, \Sigma, c)$ has a tight odd ladder  rooted at $\{u,v\}$ for  any $u\in V(P^*)\setminus \{s,t\}$ and $v\in V(G)$.
To see this, it suffices to show that for any $u\in V(G')$  $(G',\Sigma',c')$ has a 
tight odd ladder rooted at $\{e^*,u\}$.

Take any tight odd ladder ${\cal C}_s=(C_1,\dots, C_k)$ rooted at $\{s, u\}$
and any  tight odd ladder ${\cal D}_t=(D_1,\dots, D_{\ell})$  rooted at $\{t, u\}$ in $(G', \Sigma', c')$.
If $(G', \Sigma', c')$ has a tight odd cycle $C$ that contains $e^*$ and satisfies $|V(C)\cap V({\cal C}_s)|\geq 2$,
then we can get a tight odd ladder rooted at $\{e^*,u\}$  by Lemma~\ref{lem:ladder}. 
Thus we may assume that any tight odd cycle $C$ that contains $e^*$ satisfies 
$|V(C)\cap V({\cal C}_s)|=1$.
By the same reason, we may assume that any tight odd cycle $C$ that contains $e^*$
satisfies $|V(C)\cap V({\cal D}_t)|=1$.



Recall that there is at least one tight odd cycle $C$ that contains $e^*$ in $(G', \Sigma', c')$.
By the above assumption we have
\begin{equation}
\label{eq:path_reduction}
V(C)\cap V({\cal C}_s)=\{s\} \text{ and } V(C)\cap V({\cal D}_t)=\{t\}.
\end{equation}
If $|V(C_1)\cap V(D_1)|\leq 1$, then take a  path $P$ in $G[{\cal C}_s]\cup G[{\cal D}_t]$ connecting between $V(C_1)$ and $V(D_1)$ and internally disjoint from $G[C_1]\cup G[D_1]$ (where $P$ is empty if $|V(C_1)\cap V(D_1)|=1$).
Then by (\ref{eq:path_reduction}) $P\cup C_1\cup D_1\cup C$ forms a minor of the odd $K_3^2$, which is a contradiction.
Hence $|V(C_1)\cap V(D_1)|\geq 2$. 
Then by Lemma~\ref{lem:two_cycles} 
$C_1\cup D_1$ contains either an odd-subdivision of the odd-$K_4^-$ rooted at $\{s,t\}$
or a tight odd cycle passing through $s$ and $t$. 
If $C_1\cup D_1$ contains an odd-subdivision of the odd-$K_4^-$ rooted at $\{s,t\}$,
then by (\ref{eq:path_reduction}) $C_1\cup D_1\cup C$ contains an odd-subdivision of the odd-$K_4$, which is a contradiction.
Thus $C_1\cup D_1$ contains a tight odd cycle $D$ passing through $s$ and $t$.
Then Lemma~\ref{lem:two_cycles} implies that $C\cup D$ contains a tight odd cycle $D'$ with 
$e^*\in E(D')$ and  $E(D)\cap E(D')\neq \emptyset$ (and hence $(E({\cal C}_s)\cup E({\cal D}_t))\cap E(D')\neq \emptyset$).
The existence of such tight odd cycle $D'$ contradicts the assumption that any tight odd cycle that contains $e^*$ is edge-disjoint from $G[{\cal C}_s]\cup G[{\cal D}_t]$.
This completes the proof of the second claim.
\end{proof}

\subsubsection{Constructing maximum rank completions}
\label{subsub:K32}
Recall that an edge $e$ is said to be degenerate if $c(e)=\sigma(e)$.
We first solve the case when every edge is tight or nondegenerate. 

\begin{theorem}
\label{lem:l1embedding}
Let $(G,\Sigma)$ be a 2-connected signed graph which is odd-$K_3^2$ minor free and odd-$K_4$ minor free, 
and let $c\in [-1,1]^E$ with ${\rm arccos}(c)/\pi\in {\rm MET}(G,\Sigma)$. 
Suppose that every edge is tight or degenerate.
Then ${\rm P}(G,\Sigma, c)$ has a unique solution that has rank at most two.
\end{theorem}
\begin{proof}
The proof is done by induction on $|E(G)|$.
If  $(G,\Sigma)$ has a degenerate edge, we apply the contraction-resigning argument.
Specifically, if there is an even edge $e$ with $c(e)=1$,  we contract it (and remove the resulting loops); if there is an odd edge $e$ with $c(e)=-1$, we contract it after making it even by resigning. 
In the resulting problem ${\rm P}(G', \Sigma', c')$, every edge is still tight or degenerate. 
Hence by induction ${\rm P}(G', \Sigma', c')$ has a unique solution of rank  at most two. 
Let ${\rm Gram}(p)$ be this solution with $p:V(G')\rightarrow \mathbb{S}^1$. 
We then backtrack the resigning-contraction process, extending $p$ to a solution of  ${\rm P}(G, \Sigma, c)$ as follows:
if $u$ and $v$ are contracted to $w$, then let $p_u\leftarrow p_w$ and $p_v\leftarrow p_w$;
if a resigning operation was done with respect to a vertex set $S$ then let $p_v\leftarrow -p_v\ (v\in S)$.
This extension preserves uniqueness.

Therefore, we may focus on  the case when every edge is tight.
Let ${\cal C}$ be the collection of all tight odd cycles.
\begin{claim} 
${\cal H}({\cal C})$ consists of just one hyperedge spanning all the vertices.
\end{claim}
\begin{proof}
Lemma~\ref{lem:K32acyclic} implies that ${\cal H}({\cal C})$ is acyclic,
and we also have $|e\cap e'|\leq 1$ for any distinct hyperedges in ${\cal H}({\cal C})$.
Therefore if ${\cal H}({\cal C})$ has more than one hyperedge then $(G, \Sigma)$ cannot be 2-connected.
\end{proof}

If $|V(G)=2$, the the statement is trivial. 
Thus we may suppose that $|V(G)|\geq 3$.
By Lemma~\ref{lem:ear_reduction} there is a path $P$ of length at least two such that 
$P$ is contained in a tight odd cycle and ${\cal H}({\cal C}')$ consists of just one hyperedge spanning all the vertices in $G'$, where ${\cal C}'$ is the family of tight odd cycles in the path-reduction $(G', \Sigma', c')$ through $P$.
Let $s, t$ be the endvertices of $P$ and let $e_{st}$ be the edge in $E(G')\setminus E(G)$.
Without loss of generality (by resigning) we may assume that every edge in $P$  is even.

Since $|E(G')|< |E(G)|$, by induction there is a map $p':V(G')\rightarrow \mathbb{S}^1$ such that 
${\rm Gram}(p')$ is a unique solution of ${\rm P}(G', \Sigma', c')$.
 Since $e_{st}$ is contained in a tight odd cycle, we have 
 $p'(s)\cdot p'(t)=c'(e_{st})$.
 Hence 
 \begin{equation}
 \label{eq:l1_1}
 {\rm arccos}(p'(s)\cdot p'(t))={\rm arccos}(c'(e_{st}))=\sum_{e\in E(P)} {\rm arccos}(c(e))<\pi
 \end{equation}
 where the second equation follows from the definition of path-reduction and the third inequality follows from the fact that $P$ is contained in a tight cycle and $c$ is nondegenerate.
 This equation determines a unique extension of $p'$ to $p:V(G)\rightarrow \mathbb{S}^1$ such that 
 \begin{equation}
 \label{eq:em0}
 p(i)\cdot p(j)=c(ij) \quad \text{ for every $ij$ in  $P$}.
  \end{equation}
  We show that the resulting ${\rm Gram}(p)$ is feasible in ${\rm P}(G, \Sigma, c)$.
  Let $F=\{e=uv\in E(G): u\in V(P)\setminus \{s,t\}, v\in V(G)\setminus V(P)\}$.
  To see the feasibility of ${\rm Gram}(p)$, it suffices to show that 
  $\sigma(uv)p(u)\cdot p(v)\geq \sigma(uv) c(uv)$ for any edge $uv\in F$.
 
 Take any $e=uv\in F$.
 Suppose that $|F|\geq 2$. Then take an edge $e'$ from  $F\setminus \{e\}$,
 and consider ${\rm P}(G-e', \Sigma\setminus \{e'\}, c)$.
 If $(G-e, \Sigma\setminus \{e'\},c)$ contains an edge $f$ which is neither tight not degenerate.
 Then $f\in F$. We continuously increase (resp., decrease)  the value of $c(f)$ if $f$ is even (resp., odd) 
 until $f$ becomes tight or degenerate. 
 We repeatedly perform this process until every edge becomes tight or degenerate, and let $\bar{c}$ be the resulting edge vector.  
  By induction, ${\rm P}(G-e', \Sigma\setminus \{e'\}, \bar{c})$ has a solution.
  Since $G-F$ is a subgraph of $G-e'$ and ${\rm Gram}(p)$ is a unique solution of ${\rm P}(G-F, \Sigma\setminus F, c)$, ${\rm Gram}(p)$ is also a unique solution of ${\rm P}(G-e', \Sigma\setminus \{e'\}, \bar{c})$.
  Thus we get $\sigma(uv)p(u)\cdot p(v)\geq \sigma(uv) \bar{c}(uv)\geq \sigma(uv)c(uv)$ for edge $e=uv$.
  
  Therefore we may suppose $|F|=1$.
  Since every edge is tight, there is a tight odd cycle that passes through $e$. 
  Since $|F|=1$, this cycle must passes through $s$ or $t$.
  Without loss of generality, we assume that the tight cycle passes through $s$.
  Let $P_s$ be the subpaths of $P$ between $s$ and $u$.
%
Without loss of generality we assume that $e=uv$ is even.
  Let $(\tilde{G}, \Sigma')$ be the signed graph obtained from 
  $(G', \Sigma')$ by inserting a new even edge $e_{sv}$ between $s$ and $v$.
  Note that $(\tilde{G}, \Sigma')$ is still odd-$K_4$ minor and odd-$K_3^2$ minor free.
 We extend $c'$ to $\tilde{c}:E(\tilde{G})\rightarrow [-1,1]$ by setting  
 \[
 {\rm arccos}(\tilde{c}(e_{sv}))={\rm arccos}(c(uv))+\sum_{e\in E(P_s)}{\rm arccos}(c(e))
 \]
 for the new edge $e_{sv}$.
Then ${\rm arccos}(\tilde{c})/\pi \in {\rm MET}(\tilde{G}, \Sigma')$.
Hence by induction ${\rm P}(\tilde{G}, \Sigma', \tilde{c})$ is feasible.
Since ${\rm P}(G', \Sigma', c')$ has a unique solution and $G'\subseteq \tilde{G}$, ${\rm Gram}(p')$ is the solution of ${\rm P}(\tilde{G}, \Sigma', \tilde{c})$.
This in turn implies that 
\begin{equation}
\label{eq:em1}
p(s)\cdot p(v)=\tilde{c}(sv).
\end{equation}
Also, since there is a tight cycle passing through $s, u, v$ and all the edges in the path between $s$ and $v$ are even, $p(u)$ lies on the spherical line segment between $p(s)$ and $p(v)$.
In other words, 
\begin{equation}
\label{eq:em2}
{\rm arccos}(p(v)\cdot p(u))+{\rm arccos}(p(u)\cdot p(s))={\rm arccos}(p(v)\cdot p(s)).
\end{equation}
Therefore, 
\begin{align*}
{\rm arccos}(p(v)\cdot p(u))&=
{\rm arccos}(p(v)\cdot p(s))-{\rm arccos}(p(s)\cdot p(u)) & (\text{by (\ref{eq:em2})}) \\
&={\rm arccos}(p(v)\cdot p(s))-\sum_{ij\in E(P_{s})} {\rm arccos}(p(i)\cdot p(j)) &  \\
&= {\rm arccos}(\tilde{c}(vs))-\sum_{ij\in E(P_s)} {\rm arccos}(c(ij)) & (\text{by (\ref{eq:em0})})\\
&={\rm arccos}(c(vu)) & (\text{by the definition of $\tilde{c}$}).
\end{align*}
%
%
%
%
Thus $p(v)\cdot p(s)=c(vu)$ holds, and ${\rm Gram}(p)$ is feasible in ${\rm P}(G, \Sigma, c)$.
\end{proof}

Let ${\cal H}$ be a hypergraph.
We say that $v$ is a {\em cut vertex} in ${\cal H}$ if $H-v$ is disconnected for the underlying vertex-edge incidence bipartite graph $H$ of ${\cal H}$. 
If $v$ is a cut vertex, then $G$ can be decomposed into more than one hyper-subgraphs which intersect each other only at $v$.
Such a hyper-subgraph is called a {\em fraction} at $v$ if it is a member of the finest decomposition (equivalently, $v$ is no longer a cut vertex in the hyper-subgraph).

\begin{lemma}
\label{lem:K32free}
Let $(G,\Sigma)$ be a signed graph which is odd-$K_3^2$ minor free and odd-$K_4$ minor free, and let $c$ be nondegenerate.
If ${\rm arccos}(c)/\pi\in {\rm MET}(G,\Sigma)$, then ${\rm P}(G,\Sigma, c)$ is feasible and 
there is a nice dual solution supported on the strictly tight edges.
\end{lemma}
\begin{proof}
Let ${\cal C}$ be the set of tight odd cycles in $(G,\Sigma)$ with respect to $c$, and let $T$ be the set of tight edges. 
By restricting the problem to each tight odd cycle $C$, Lemma~\ref{cycle} gives a dual solution $\Omega_C$ of corank two and supported on $E(C)$.
Let $\Omega=\sum_{C\in {\cal C}} \Omega_C$ by regarding each $\Omega_C$ as a matrix of size $|V(G)|\times |V(G)|$.
Our goal is to find a solution $X$ that satisfies the  strict complementarity condition with $\Omega$.

We shall continuously change the value of $c(e)$ for all non-tight and non-degenerate edges $e$ such that $c(e)$ monotonically increases (resp., decreases) if $e$ is even (resp., odd), until an edge becomes tight or degenerate.
We keep this process until every edge becomes tight or degenerate,
and let $c'$ be the resulting vector. 
Note that ${\rm arccos}(c')/\pi\in {\rm MET}(G,\Sigma)$.
Note also that $\sigma(e)c'(e)>\sigma(e)c(e)$ for every $e\in E(G)\setminus T$.
By Theorem~\ref{lem:l1embedding} there is a solution $X'$ of ${\rm P}(G,\Sigma, c')$ with rank at most two.
(If $G$ is not 2-connected, then one can apply Theorem~\ref{lem:l1embedding} to each 2-connected component and then combine the solutions of the 2-connected components 
to get  a solution of ${\rm P}(G,\Sigma, c')$.)  
This $X'$ is also a solution of ${\rm P}(G,\Sigma,c)$ and satisfies the  strict complementarity condition with $\Omega$.

Take $p':V(G)\rightarrow \mathbb{S}^d$ such that $X'={\rm Gram}(p')$.
Also let $\epsilon$ be a positive number.
By Lemma~\ref{lem:K32acyclic} ${\cal H}({\cal C})$ is acyclic. 
Therefore, by (slightly) rotating points $p'(v)$ of each fraction at each cut vertex of ${\cal H}({\cal C})$,  one can get $p:V(G)\rightarrow \mathbb{S}^{n-1}$ satisfying the following properties.
\begin{description}
\item[(i)] $p_i\cdot p_j=p_i'\cdot p_j'$ 
if $i$ and $j$ belong to a hyperedge of ${\cal H}({\cal C})$.
\item[(ii)] $|p_i\cdot p_j-p_i'\cdot p_j'|\leq \epsilon$ for every $i,j\in V(G)$.
\item[(iii)] For each cut vertex $v$ of ${\cal H}({\cal C})$ and distinct fractions $H$ and $H'$ at $v$, 
$(\spa p(V(H)))\cap (\spa p(V(H')))=\spa p(v)$.
\item[(iv)]  For any distinct connected components $C$ and $C'$ in ${\cal H}({\cal C})$, 
$(\spa p(V(C)))\cap (\spa p(V(C')))=\{0\}$.

\end{description}

Let $X={\rm Gram}(p)$.
(i) implies that $X\Omega=X'\Omega=0$.
Also (i) implies that $p_i\cdot p_j=p_i'\cdot p_j'$ for every $ij\in T$.
On the other hand, since $\sigma(e)c'(e)>\sigma(e)c(e)$ for $e\in E(G)\setminus T$, if $\epsilon$ is sufficiently small, then 
(ii) implies $\sigma(ij)  (p_i\cdot p_j) \geq \sigma(ij) (p_i' \cdot p_j')-\epsilon
\geq \sigma(ij) c'(ij)-\epsilon> \sigma(ij) c(ij)$ for every $ij\in E(G)\setminus T$.
Thus $X$ is a feasible solution for ${\rm P}(X, \Sigma, c)$.

(iii)(iv) imply that $\rank X=|E({\cal H}({\cal C}))|+\omega$, where $\omega$ denotes the number of connected components  in ${\cal H}({\cal C})$.
On the other hand, by applying Lemma~\ref{lem:glueing} inductively (following the acyclic structure of ${\cal H}({\cal C})$), we have that 
the corank of $\Omega$ is equal to $|E({\cal H}({\cal C}))|+\omega$. 
Therefore $(X, \Omega)$ satisfies the  strict complementarity condition.
\end{proof}

\subsection{The remaining case}
\label{subsec:proof}

The remaining for the proof of Theorem~\ref{thm:oddK4} is the case when 
$(G,\Sigma)$ has a cut vertex or a strong 2-split. 

\begin{proof}[Proof of Theorem~\ref{thm:oddK4}]
The proof is done by induction on the number of vertices.  
By Lemma~\ref{K32} and Lemma~\ref{lem:K32free}, we may assume that 
$(G,\Sigma)$ contains an odd-$K_3^2$ minor but is not equal to the odd-$K_3^2$.
By Theorem~\ref{thm:structural} $(G,\Sigma)$ has a cut vertex or a strong 2-split.

The statement easily follows from Lemma~\ref{lem:glueing} if  $(G,\Sigma)$ has a cut vertex.
Hence we may assume that $G$ is 2-connected and  $(G,\Sigma)$ has a strong 2-split.
Let $(G_i,\Sigma_i) \ (i=1,2)$ be the parts of the strong 2-split.
The vertices of $V(G_1)\cap V(G_2)$ are denoted by $u$ and $v$, and for $i=1, 2$ 
the new odd (resp., even) edge between $u$ and $v$ in $(G_i, \Sigma_i)$ is denoted by $f_{i, o}$ (resp., $f_{i,e}$).
Also let $(\tilde{G}_i, \tilde{\Sigma}_i)=(G_i- f_{i,e} -f_{i,o}, \Sigma_i-f_{i,o})$, which is a subgraph of $(G,\Sigma)$. 
Note that $(\tilde{G}_i,\tilde{\Sigma}_i)$ is connected and is not bipartite (by the definition of strong 2-split). 
Hence the 2-connectivity of  $G$ implies that $(\tilde{G}_i,\tilde{\Sigma}_i)$ has both  an even path and an odd path, respectively, between $u$ and $v$.

Let $x={\rm arccos}(c)/\pi \in {\rm MET}(G,\Sigma)$.
By $c\in (-1,1]^{E\cap \Sigma}\times [-1,1)^{E\setminus \Sigma}$, we have 
\begin{equation}
\label{eq:x_assumption}
x\in [0,1)^{E\cap \Sigma}\times (0,1]^{E\setminus \Sigma}.
\end{equation}
For each path $P$, we define $\lambda(P)\in \mathbb{R}$ by
\[
\lambda(P)=\begin{cases}
{\rm val}(P,x) & \text{ if $P$ is even} \\
1-{\rm val}(P,x) & \text{ if $P$ is odd}
\end{cases}
\] 
and for $i=1,2$ let 
\begin{equation}
\label{eq:lambda}
\begin{split}
\lambda_{i,e}&=\min\{\lambda(P): P \text{ is an even path between $u$ and $v$ in $(\tilde{G}_i, \tilde{\Sigma}_i)$}\} \\
\lambda_{i,o}&=\max\{\lambda(P): P \text{ is an odd path between $u$ and $v$ in $(\tilde{G}_i, \tilde{\Sigma}_i)$}\}.
\end{split}
\end{equation}
Since $x\in {\rm MET}(G,\Sigma)$, we have 
\begin{equation}
\label{eq:lambda1}
\lambda_{i,o}\leq \lambda_{j,e} \qquad (1\leq i, j\leq 2).
\end{equation}
Also by (\ref{eq:x_assumption}) we have
\begin{equation}
\label{eq:lambda2}
\lambda_{i,e}>0 \quad \text{and} \quad \lambda_{i,o}<1.
\end{equation}
These inequalities imply that
\[
I:=\{\lambda\in (0,1) : \lambda_{i,o}\leq \lambda \leq \lambda_{i,e} \ (i=1,2)\}
\]
is nonempty.
We take any number $\lambda^*$ from the relative interior of $I$.

Define $c_i\in [-1,1]^{E(G_i)}$ such that  $c_i(e)=c(e)$ for $e\in E(\tilde{G}_i)$ and 
$c_i(f_{i,e})=c_i(f_{i,o})=\cos(\pi \lambda^*)$.

\begin{claim}
\label{claim:main_cycle}
The following hold.
\begin{itemize}
\setlength{\parskip}{0.1cm} 
 \setlength{\itemsep}{0.1cm} 
\item For each $i$, ${\rm arccos}(c_i)/\pi\in {\rm MET}(G, \Sigma)$.
\item $\lambda^*=\lambda_{i,o}$ if and only if $(G_i,\Sigma_i)$ has a tight odd cycle $C_i$  that passes through $f_{i,e}$ and not through $f_{i,o}$.
\item $\lambda^*=\lambda_{i,e}$ if and only if $(G_i,\Sigma_i)$ has a tight odd cycle $C_i$ that passes through $f_{i,o}$ and not through $f_{i,e}$.
\item Suppose that  $(G_i, \Sigma_i)$ has a tight odd cycle $C_i$  passing through $f_{i,e}$ and
$(G_j, \Sigma_j)$ has  a tight odd cycle $C_j$   passing through $f_{j,o}$.
Then $(C_i\setminus \{f_{i,e}\})\cup (C_j\setminus \{f_{j,o}\})$ forms a tight odd cycle. 
\end{itemize}
\end{claim}
\begin{proof}
The first claim follows from $\lambda_{i,o}\leq \lambda^*\leq \lambda_{i,e}$ and the definition of $c_i(f_{i,e})$ and $c_i(f_{i,o})$.

To see the second claim, suppose that $\lambda^*=\lambda_{i,o}$.
Then $(\tilde{G}_i, \tilde{\Sigma}_i)$ has an odd path $P$ between $u$ and $v$ 
with $x(f_{i,e})=1-{\rm val}(P,x)$,
or equivalently ${\rm val}(P\cup \{f_{i,e}\},x)=1$.
Hence $P\cup \{f_{i,e}\}$ is a tight odd cycle in $(G_i, \Sigma_i)$.
Reversing the argument, we also see  the converse direction of the second claim.

The third claim follows from the same argument as the second one.

The fourth claim can be seen as follows.
Let $C=(C_i\setminus \{f_{i,e}\})\cup (C_j\setminus \{f_{j,o}\})$.
By the first and the second claims, we have $\lambda_{i,e}=\lambda^*=\lambda_{j,o}$.
Therefore, since $c_i(f_{i,e})=\cos (\pi \lambda^*)=c_j(f_{j,e})$, we have
\begin{equation}
\label{eq:main_cycle2}
{\rm val}(E(C_i)\setminus \{f_{i,e}\}, x)+{\rm val}(E(C_j)\setminus \{f_{j,o}\},x)=1.
\end{equation}
Thus, what remains is to verify that $C$ forms an odd cycle.
If $i\neq j$, then $C$ is clearly an odd cycle.
Suppose $i=j$.
Since $C$ is Eulerian (by replacing each edge in $(E(C_i)\setminus \{f_{i,e}\})\cap (E(C_j)\setminus \{f_{j,o}\})$ by parallel edges), it can be decomposed into edge-disjoint cycles $E_1,\dots, E_\ell$.
By  (\ref{eq:x_assumption}),
 ${\rm val}(E_j, x)>0$ if $E_j$ is even,
and ${\rm val}(E_j, x)\geq 1$ if $E_j$ is odd for each $1\leq j\leq \ell$.
Since there is at least one odd cycle, by (\ref{eq:main_cycle2}) we get
$1={\rm val}(E(C_i)\setminus \{f_{i,e}\}, x)+{\rm val}(E(C_j)\setminus \{f_{j,o}\},x)
=\sum_{1\leq j\leq \ell}{\rm val}(E_j,x)\geq 1$,
meaning that $C$ consists of just one odd cycle, which is odd. 
\end{proof}

By Claim~\ref{claim:main_cycle}  $c_i\in {\rm MET}(G, \Sigma)$, and $c_i$ is nondegenerate by (\ref{eq:x_assumption}). Hence we can apply the induction hypothesis to get  a nice dual solution $\omega_i$ for each $(\tilde{G}_i, \tilde{\Sigma}_i, c_i)$ supported on the strictly tight edges.
We take such $\omega_i$ so that it satisfies an extra property as follows.
\begin{claim}
\label{claim:main_cycle3}
There is a nice dual solution $\omega_i\ (i=1,2)$ supported on the strictly tight edges such that 
\begin{equation}
\label{eq:main_cycle3}
\omega_1(f_{1,e})+\omega_1(f_{1,o})+\omega_2(f_{2,e})+\omega_2(f_{2,o})=0.
\end{equation}
\end{claim}
\begin{proof}
The proof is split into three cases.

(Case 1) If $\lambda_{i,e}<\lambda_{j,o}$ for any $i, j\in \{1,2\}$, then 
Claim~\ref{claim:main_cycle} implies that 
$(G_i, \Sigma_i)$ has no tight cycle $C_i$  passing through $f_{i,e}$ nor $f_{i,o}$.
Hence $\omega_i(f_{i,e})=\omega_i(f_{i,o})=0$ and (\ref{eq:main_cycle3}) follows.

(Case 2)
If $\lambda_{i,e}=\lambda_{j,o}$ for  $i\neq j$, then 
by Claim~\ref{claim:main_cycle}
$(G_i, \Sigma_i)$ has a tight cycle $C_i$  passing through $f_{i,e}$
and $(G_j, \Sigma_j)$ has  a tight cycle $C_j$   passing through $f_{j,o}$.
We can suppose that $\omega_i(f_{i,e})<0$ $\omega_i(f_{i,o})=0$.
(To see this, suppose $\omega_1(f_{i,o})\neq 0$. 
We can assume $|\omega_i(f_{i,e})|>|\omega_i(f_{i,o})|$, since otherwise we can increases $|\omega_i(f_{i,e})|$  by adding a dual feasible solution supported on $E(C_i)$. 
Setting $\omega_i(f_{i,e})\leftarrow \omega_i(f_{i,e})+\omega_i(f_{i,o})$ and 
$\omega_i(f_{i,o})\leftarrow 0$ we have a nice dual solution such that $\omega_i(f_{i,o})=0$.)
Similarly we can suppose that $\omega_j(f_{j,o})<0$ and $\omega_j(f_{j,e})=0$.
By multiplying $\omega_2$ by a positive number, we may further suppose 
$\omega_i(f_{i,e})+\omega_j(f_{j,o})=0$. Then the resulting vectors satisfy (\ref{eq:main_cycle3}).

(Case 3) Suppose $\lambda_{1,e}=\lambda_{1,o}$ or $\lambda_{2,e}=\lambda_{2,o}$.
Without loss of generality assume $\lambda_{1,e}=\lambda_{1,o}$.
Then $(G_1, \Sigma_1)$ has a tight cycle $C_e$  passing through $f_{1,e}$
and a tight cycle $C_o$  passing through $f_{1,o}$.
We may suppose that $\omega_1(f_{1,o})=\omega_1(f_{1,e})=0$
(otherwise, we first change $\omega_1$ so that 
$|\omega_1(f_{1,o})|=|\omega_1(f_{1,e})|$ 
by adding dual feasible solutions supported on $E(C_o)$ or $E(C_e)$ 
and apply  the canceling of $\omega_1(f_{1,o})$ and $\omega_1(f_{1,e})$
so that  $\omega_1(f_{1,o})=\omega_1(f_{1,e})=0$).
Similarly we can suppose $\omega_2(f_{2,o})=\omega_2(f_{2,e})=0$ if $\lambda_{2,o}=\lambda_{2,e}$.
If  $\lambda_{2,o}\neq \lambda_{2,e}$, then by Case 2 we can assume
$\lambda_{2,o}\neq \lambda_{1,e}$ and $\lambda_{2,e}\neq \lambda_{1,o}$,
and hence $\lambda_{2,o}<\lambda^*< \lambda_{2,e}$. This implies that neither $f_{2,o}$ nor $f_{2,e}$ is strictly tight, and we again have $\omega_2(f_{2,o})=\omega_2(f_{2,e})=0$. This confirms (\ref{eq:main_cycle3}).
\end{proof}

(\ref{eq:main_cycle3}) implies that there is a dual solution $\omega$ of ${\rm P}(G,\Sigma, c)$ 
such that $\Omega=\Omega_1+\Omega_2$. 
Lemma~\ref{lem:glueing} implies that $\Omega_1+\Omega_2$ satisfies the  strict complementarity condition with any maximum rank solution of ${\rm P}(G,\Sigma, c)$.
(Note that, for any solutions $X_i$ in ${\rm P}(G_i, \Sigma_i, c_i)$, we have $X_1[u,v]=\lambda^*=X_2[u,v]$.)
 
It remains to show that the resulting vector $\omega$ is supported on the strictly tight edges.
Take any $e\in E(G)$ with $\omega(e)\neq 0$, and without loss of generality assume $e\in E(\tilde{G}_1)$.
Since $\omega(e)\neq 0$, $e$ is contained in a tight cycle $C$ of length at least three in $(G_1, \Sigma_1)$.
If $C$ contains neither $f_{1,e}$ nor $f_{1,o}$, then $C$ remains a tight cycle in $(G,\Sigma)$,
and hence $e$ is strictly tight in $(G,\Sigma)$.
If $C$ contains, say $f_{1,e}$, then by Claim~\ref{claim:main_cycle} 
$\lambda^*=\lambda_{1,o}$.
Since $\lambda^*$ is taken to be in the relative interior of $I$, 
we must have $\lambda_{i,o}=\lambda_{1,e}$ for some $i\in \{1,2\}$ by Claim~\ref{claim:main_cycle}.
Therefore, by Claim~\ref{claim:main_cycle} again, 
$(G_i, \Sigma_i)$ contains a tight odd cycle $C_i$ passing through $f_{i,o}$ 
and moreover $(C\setminus \{f_{1,e}\})\cup (C_i\setminus \{f_{i,o}\})$ is a tight odd cycle in $(G,\Sigma)$.
Thus $e$ is strictly tight in $(G,\Sigma)$. This completes the proof of Theorem~\ref{thm:oddK4}.
\end{proof}

\subsection{Computational aspect}
\label{subsec:algorithms}
We give a  brief remark on the computational aspect.
Although the metric polytope consists of exponential number of inequalities, Barahona and Mahjoub~\cite{bm} showed that one can decide whether $x\in \mathbb{R}^{E}$ belongs to ${\rm MET}(G)$  in polynomial time by computing the shortest distances of $O(|V|)$ pairs in an auxiliary weighted graph.
(See, e.g., \cite[Section~27.3.1]{dl}.)
The technique can be trivially adapted to signed graphs. 
Thus, provided that ${\rm arccos}(c)/\pi$ is given as input, one can check wether ${\rm P}(G, \Sigma, c)$ is feasible or not by Theorem~\ref{thm:oddK4projection}.

Our proof of Theorem~\ref{sec:proof} is constructive and computes a maximum rank completion in polynomial time (in the real RAM model). 
We sketch the algorithm:
\begin{itemize}
\item By computing a shortest path in the above auxiliary graph, 
one can compute 
\begin{align*}
\lambda_{u,v,e}&:=\min\{{\rm val}(P, {\rm arccos}(c)/\pi): P \text{ is an even path between $u$ and $v$}\} \\
\lambda_{u,v,o}&:=\min\{{\rm val}(P, {\rm arccos}(c)/\pi): P \text{ is an odd path between $u$ and $v$}\}
\end{align*}
  for any pair of vertices $u, v$ in polynomial time. 
Hence, if $(G,\Sigma)$ has a strong 2-split, then a solution of ${\rm P}(G, \Sigma, c)$ can be constructed from a solution of each part of the 2-split as shown in Section~\ref{subsec:proof}.

\item  From $\lambda_{u,v,e}$ and $\lambda_{u,v,o}$,  one can decide whether there is a tight odd cycle that contains $u$ and $v$ if $c$ is nondegenerate (cf.~the fourth claim of Claim~\ref{claim:main_cycle}), and hence one can compute ${\cal H}({\cal C})$ for the family of tight odd cycles in polynomial time. 
Also one can decide whether an edge is tight or not in polynomial time.

\item By Theorem~\ref{thm:structural}, if $(G,\Sigma)$ has neither a cut vertex nor a strong 2-split, then it is equivalent to the odd-$K_3^2$ or odd-$K_3^2$ minor free.

\item If $(G, \Sigma)$ is equivalent to the odd-$K_3^2$ or it is odd-$K_3^2$ minor free, then we use  the construction of a completion given the proof of Lemma~\ref{K32} or that of Lemma~\ref{lem:K32free}, respectively. We remark that, denoting by $\bar{T}$ the set of edges which are neither tight nor degenerate,   the shortest path technique can also computes $\max\{\varepsilon\in \mathbb{R}: x+\sum_{e\in \bar{T}} \sigma(e) (\varepsilon \chi_e)\in {\rm MET}(G, \Sigma)\}$, where $\chi_e$ denotes the characteristic vector of $e$ in $\mathbb{R}^E$.
Thus the construction given in the proof of Lemma~\ref{K32} or that of Lemma~\ref{lem:K32free} can be implemented.
\end{itemize}

The proof of Lemma~\ref{lem:K32free} also implies that, if ${\rm P}(G,\Sigma,c)$ is feasible and $(G,\Sigma)$ is odd-$K_4$ minor and odd-$K_3^2$ minor free, then there always exists a solution of rank at most two. This solution can be computed in polynomial time in the real RAM model.
It is known that the problem of deciding whether a PSD matrix completion problem has a solution of rank at most two is NP-hard, even if the underlying graph is a cycle $C_n$~\cite{nlv}.
This means that, in the signed setting, the problem is NP-hard even if the underlying signed graph is restricted to the odd-$C_n^2$.

Theorem~\ref{thm:structural} implies that any odd-$K_4$ minor free signed graph can be decomposed into copies of the odd-$K_3^2$ and odd-$K_3^2$-minor free signed graphs by strong 2-splits.
Therefore,  if ${\rm P}(G,\Sigma,c)$ is feasible and $(G,\Sigma)$ is odd-$K_4$ minor, then there always exists a solution of rank at most three and a solution can be computed in polynomial time.

\section{Unique Completability}
\label{sec:unique}
In this section we show that the proof  of Theorem~\ref{thm:oddK4} can be adapted to characterizing the unique solvability of completion problems.

Given a signed graph $(G, \Sigma)$ and a set ${\cal C}$ of odd cycles in $(G,\Sigma)$, 
recall that the hypergraph ${\cal H}({\cal C})$ on $V(G)$ is defined as one obtained by regarding each cycle in ${\cal C}$ as a hyperedge and repeatedly merging a pair of hyperedges $e, e'$ with $|e\cap e'|\geq 2$ into a single hyperedge.
If  $c$ is nondegenerate, our characterization  will be given in term of hypergraph ${\cal H}({{\cal C}_c})$, where ${\cal C}_c$ is the set of tight odd cycles with respect to $c$.
If $(G,\Sigma, p)$ is degenerate, we can reduce the problem to the nondegenerate case by using contraction and resigning.  
Specifically, we first compute $(G', \Sigma', c')$ from $(G,\Sigma, c)$ by a sequence of resigning and contraction of an even edge $ij$ with $c(ij)=1$ such that $c'$ is nondegenerate,
and define a hypergraph ${\cal H}(G,\Sigma, c)$ to be ${\cal H}({\cal C}_{ c' })$.
Note that ${\cal H}(G,\Sigma, c)$ is uniquely defined by $(G,\Sigma, c)$.

We say that a hypergraph forms a {\em triangle} if 
there is no isolated vertex and 
it consists of three edges $e_1, e_2, e_3$
such that each pair of them intersects at exactly one vertex.
Lemma~\ref{lem:l1embedding} imply the following.
\begin{lemma}
\label{lem:unique}
Let $(G,\Sigma)$ be a signed graph and let $c\in [-1,1]^E$ with $c\in {\cal E}(G,\Sigma)$. Then the following holds:
\begin{itemize}
\setlength{\parskip}{0.1cm} 
 \setlength{\itemsep}{0.1cm} 
\item If ${\cal H}(G,\Sigma, c)$ has only one vertex, 
then ${\rm P}(G,\Sigma, c)$ is uniquely solvable and its solution has rank one.
\item If ${\cal H}(G,\Sigma, c)$ has a hyperedge spanning all the vertices, 
then ${\rm P}(G,\Sigma, c)$ is uniquely solvable and its solution has rank at most two.
\item If ${\cal H}(G,\Sigma, c)$ forms a triangle, 
then ${\rm P}(G,\Sigma, c)$ is uniquely solvable and its solution has rank at most three.
\end{itemize} 
\end{lemma}

The reverse implication is not true in general but turns out to be true if $(G,\Sigma)$ is odd-$K_4$ minor free. We first give a characterization for odd-$K_4$ minor and odd-$K_3^2$ minor free graphs.
\begin{theorem}
\label{thm:unique_K32}
Let $(G,\Sigma)$ be a signed graph and let $c\in [-1,1]^E$ with ${\rm arccos}(c)/\pi\in {\rm MET}(G, \Sigma)$. 
Suppose that $(G, \Sigma)$ is odd-$K_4$ minor and odd-$K_3^2$ minor free. 
Then the following holds:
\begin{itemize}
\setlength{\parskip}{0.1cm} 
 \setlength{\itemsep}{0.1cm} 
\item If ${\cal H}(G,\Sigma, c)$ has only one vertex, 
then ${\rm P}(G,\Sigma, c)$ is uniquely solvable and its solution has rank one.
\item If ${\cal H}(G,\Sigma, c)$ has a hyperedge spanning all the vertices, 
then ${\rm P}(G,\Sigma, c)$ is uniquely solvable and its solution has rank at most two.
\item Otherwise, ${\rm P}(G,\Sigma, c)$ is not uniquely solvable.
\end{itemize} 
\end{theorem}
\begin{proof}
By Lemma~\ref{lem:unique} we need to show that, if ${\cal H}(G,\Sigma, c)$ has an isolated vertex or 
more than one hyperedge, then ${\rm P}(G,\Sigma, c)$ is not uniquely solvable. 
Since the unique solvability is invariant under reduction 
(that is, a sequence of resining and contraction of degenerate even edges) we may focus on the case when 
$c$ is nondegenerate. 
Then we have ${\cal H}(G,\Sigma, c)={\cal H}({\cal C}_c)$.

Construct $c'$ such that 
\begin{equation}
\label{eq:universal1}
c'(ij)=\begin{cases}
c(ij) & \text{ if $i$ and $j$ belong to a hyperedge in ${\cal H}({\cal C}_c)$} \\
c(ij)-\sigma(ij)\epsilon & \text{ otherwise }
\end{cases} \qquad (ij\in E(G))
\end{equation}
for some positive number $\epsilon$ which will be specified later.
Let ${\cal C}'$ be the set of tight odd cycles with respect to $c'$.
If $\epsilon>0$ is sufficiently small, ${\cal C}_c={\cal C}'$
and ${\rm arccos}(c')/\pi \in {\rm MET}(G, \Sigma)$.
Hence by Theorem~\ref{thm:oddK4projection} there is $q:V(G)\rightarrow \mathbb{S}^{d'}$ for some $d'>0$ such that ${\rm Gram}(q)$ is a solution of ${\rm P}(G, \Sigma, c')$.

Since $(G,\Sigma)$ is odd-$K_4$ minor and odd-$K_3^2$ minor free, 
Lemma~\ref{lem:K32acyclic} implies that ${\cal H}({\cal C}')$ is acyclic.
Hence, taking a cut vertex $v$ of ${\cal H}({\cal C}')$,
we can continuously rotate the points $q(u)$ in each fraction of ${\cal H}({\cal C}')$ at the cut vertex $v$ 
so that 
\begin{itemize}
\setlength{\parskip}{0.1cm} 
 \setlength{\itemsep}{0.1cm} 
\item $q(i)\cdot q(j)$ remains unchanged if $i$ and $j$ belong to a  hyperedge,
and 
\item $q(i)\cdot q(j)$ changes for some pair $i$ and $j$.
\end{itemize}
Due to the triangle inequality in spherical space, we can control the change of 
$|q(i)\cdot q(j)|$ by arbitrary small positive number $\epsilon'$.
Therefore, if $\epsilon'$ is taken to be $\epsilon'<\epsilon$, then the resulting ${\rm Gram}(q)$ is feasible in  ${\rm P}(G,\Sigma, c)$. 
Thus there are more than one solution in ${\rm P}(G,\Sigma, c)$.
\end{proof}

\begin{theorem}
\label{thm:unique_K4}
Let $(G,\Sigma)$ be a signed graph and let $c\in [-1,1]^E$ with ${\rm arccos}(c)/\pi\in {\rm MET}(G, \Sigma)$. 
Suppose that $(G, \Sigma)$ is odd-$K_4$ minor free. 
Then the following holds:
\begin{itemize}
\setlength{\parskip}{0.1cm} 
 \setlength{\itemsep}{0.1cm} 
\item If ${\cal H}(G,\Sigma, c)$ has only one vertex, 
then ${\rm P}(G,\Sigma, c)$ is uniquely solvable and its solution has rank one.
\item If ${\cal H}(G,\Sigma, c)$ has a hyperedge spanning all the vertices, 
then ${\rm P}(G,\Sigma, c)$ is uniquely solvable and its solution has rank at most two.
\item If ${\cal H}(G,\Sigma, c)$ forms a triangle, 
then ${\rm P}(G,\Sigma, c)$ is uniquely solvable and its solution has rank at most three.
\item Otherwise, ${\rm P}(G,\Sigma, c)$ is not uniquely solvable.
\end{itemize} 
\end{theorem}
\begin{proof}
We may again assume that $c$ is nondegenerate, and  ${\cal H}(G,\Sigma, c)={\cal H}({\cal C}_c)$.
The proof is done by induction on the number of vertices.
By Theorem~\ref{thm:universal_K32} and Lemma~\ref{lem:K32acyclic}, 
the statement follows if $(G,\Sigma)$ is odd-$K_3^2$ minor free.

Suppose that $(G,\Sigma)$ is odd-$K_3^2$.
If there is a tight triangle with respect to $c$, then ${\rm P}(G,\Sigma,c)$ has a unique rank-two solution by Lemma~\ref{K3}, implying the second case of the statement.
If there is no tight triangle but each parallel class forms a tight cycle, then the third case of the statement holds.
If none of them is applicable, then in the proof of Lemma~\ref{K32} we have seen that ${\rm P}(G, \Sigma, c)$ has more than one solution. Hence the statement follows.

Thus we  assume that $(G,\Sigma)$ is odd-$K_3^2$ minor free and is not equivalent to odd-$K_3^2$.
Then by Theorem~\ref{thm:structural} $(G, \Sigma)$ has a cut vertex or a strong 2-split.
If $(G, \Sigma)$ has a cut vertex, then ${\rm P}(G, \Sigma, c)$ is  not uniquely solvable.
Hence we focus on the case when $(G, \Sigma)$ has a strong 2-split.

From Lemma~\ref{lem:unique} we need to show that, if none of the first three conditions are satisfied, then ${\rm P}(G,\Sigma, c)$ is not uniquely solvable. Suppose for a contradiction that ${\rm P}(G, \Sigma, c)$ is uniquely solvable.
Let ${\rm Gram}(p)$ be the solution for some $p:V(G)\rightarrow \mathbb{S}^d$.
Let $(G_1, \Sigma_1)$ and $(G_2, \Sigma_2)$ are the parts of the strong 2-split,
and denote $V(G_1)\cap V(G_2)=\{u,v\}$.
Let $p_i$ be the restriction of $p$ to $V(G_i)$.
We may assume that $p(u)\neq p(v)$, since otherwise 
${\rm P}(G, \Sigma, c)$ has more than one solution since 
$\{u,v\}$ is a cut set of $G$ and $|V(G_i)|\geq 3$.
Hence, setting $c_i\in \pi_{G_i}({\rm Gram}(p_i))$,  each $c_i$ is nondegenerate for each $i=1,2$.

Since $G_i$ contains parallel edges between $u$ and $v$,  for any solution $X_i$ of ${\rm P}(G_i, \Sigma_i, c_i)$ we have  $X_1[u,v]=p(u)\cdot p(v)=X_2[u,v]$, and hence $X_1$ and $X_2$
can be glued together to be a solution $X$ of ${\rm P}(G, \Sigma, c)$.
The resulting solution $X$ can be distinct from ${\rm Gram}(p)$ if 
both $X_1$ and $X_2$ have rank at least three.
Hence we may assume that $X_1$ has rank two.
In the same reason,  we may assume that each ${\rm P}(G_i, \Sigma_i, c_i)$ has a unique solution
(since otherwise we can get more than one solution of ${\rm P}(G, \Sigma, c)$ by glueing).
Hence by induction
${\cal H}(G_1, \Sigma_1, c_1)$ consists of single hyperedge spanning all the vertices,
and ${\cal H}(G_2, \Sigma_2, c_2)$ consists of single hyperedge  or forms a triangle.

The remaining of the proof follows the proof of Theorem~\ref{thm:oddK4}.
Recall that $(G_i, \Sigma_i)$  has an odd edge $f_{i,o}$ and an even edge $f_{i,e}$  
between $u$ and $v$, which are not contained in $(G,\Sigma)$. Let $(\tilde{G}_i, \tilde{\Sigma}_i)=(G_i-f_{i,e}-f_{i,o}, \Sigma_i-f_{i,o})$.
We use $\lambda_{i,o}$ and $\lambda_{i,e}$ defined in (\ref{eq:lambda}).
Also let $\lambda^*$ be any number  in the relative interior of $I:=\{\lambda\in (0,1): \lambda_{i,o}\leq \lambda\leq \lambda_{i,e}\}$.
The proof is split into three cases  depending on the values of $\lambda_{i,o}$ and $\lambda_{i,e}$.

(Case 1) Suppose that $\lambda_{i,o}=\lambda_{j,e}$ for some $i,j\in\{1,2\}$.
By (\ref{eq:lambda1}), $\lambda_{i,o}=\lambda^*=\lambda_{j,e}$. 
Hence, by Claim~\ref{claim:main_cycle}, 
$(G_i, \Sigma_i)$ has a tight odd cycle $C_o$ that passes through $f_{i,o}$ but not through $f_{i,e}$,
$(G_j, \Sigma_j)$ has a tight odd cycle $C_e$ that passes  through $f_{j,e}$ but not through $f_{j,o}$, and 
$C:=(C_o\setminus \{f_{i,o}\})\cup (C_e\setminus \{f_{j,e}\})$ is a tight odd cycle  in $(G,\Sigma)$.
Let $e^*$ be a hyperedge in ${\cal H}(G_2, \Sigma_2, c_2)$ to which $u$ and $v$ are belonging
(which exists since $\{f_{2,o}, f_{2,e}\}$ forms a tight cycle).
Then  the existence of $C$ implies that ${\cal H}(G,\Sigma,c)$ is obtained from ${\cal H}(G_2,\Sigma_2,c_2)$ 
by replacing $e^*$ with $V(G_1)\cup e^*$.
(To see this, consider the process of constructing ${\cal H}(G_i, \Sigma_i, c_i)$ that involves the tight cycle $\{f_{i,o}, f_{i,e}\}$.
We replace $\{f_{i,o}, f_{i,e}\}$ with $C$ in each process to get a partial process for constructing ${\cal H}(G,\Sigma,c)$. Concatenating those two partial processes,  we get a process of constructing ${\cal H}(G,\Sigma,c)$.
The process ends up with a hypergraph where $e^*$ will be replaced with $V(G_1)\cup e^*$ in ${\cal H}(G_2,\Sigma_2,c_2)$ since ${\cal H}(G_1, \Sigma_1, c_1)$ consists of just one hyperedge.)
Therefore ${\cal H}(G,\Sigma,c)$ consists of a single edge or forms a triangle,
which is a contradiction.

(Case 2) Suppose that  $\lambda_{i,o}<\lambda_{j,e}$ for any $i, j\in \{1,2\}$.
In this case we have $\lambda_{i,o}<\lambda^*<\lambda_{i,e}$
for every $i\in \{1,2\}$.
Moreover, we can take $\lambda^*$  such that 
$\lambda^*\neq p(u)\cdot p(v)$ 
since $I$ forms an interval.
Let $(G', \Sigma')$ be a signed graph obtained from $(G, \Sigma)$ by adding two parallel edges $e_1$ and $e_2$ between $u$ and $v$ with distinct signs, and extend $c\in [0,1]^{E(G)}$ to $c'\in [0,1]^{E(G')}$ 
by setting $c'(e_1)=c'(e_2)=\lambda^*$.
Due to the definition of $\lambda^*$, ${\rm arccos}(c')/\pi\in {\rm MET}(G', \Sigma')$.
Moreover, $(G', \Sigma')$ is still odd-$K_4$ minor free, since $\{u,v\}$ is a cut set.
Hence by Theorem~\ref{thm:oddK4projection} 
${\rm P}(G', \Sigma', c')$ has a solution $Y$.
Since $(G,\Sigma)$ is a subgraph of $(G', \Sigma')$, $Y$ is a solution of  ${\rm P}(G, \Sigma, c)$
and $Y[u,v]=c'(uv)\neq p(u)\cdot p(v)$.
Therefore ${\rm P}(G', \Sigma', c')$ has more than one solution, implying that $(G,p)$ is not universally rigid. 
\end{proof}

As remarked in Section~\ref{subsec:algorithms}, ${\cal H}({\cal C})$ can be computed in polynomial time if ${\rm arccos}(c)/\pi$ is given as input. Hence the condition can be tested.
We also remark that the PSD matrix completion problem cannot be uniquely solvable  if a diagonal entry is missing.
For the unique solvability of the rank-constrained matrix completion problem, see, e.g., \cite{jjt,ktt}.


\section{Characterizing Universal Rigidity of Odd-$K_4$ Minor Free Graphs}
\label{sec:universal}
A {\em spherical tensegrity} is a tuple $(G, \Sigma, p)$ of a signed graph $(G, \Sigma)$ and $p:V\rightarrow \mathbb{S}^d$ for some integer $d$.
One central problem in the rigidity theory is to decide the {\em global rigidity} of tensegrities (see, e.g., \cite{c14} for the definition).
As a relaxation of global rigidity, Zhu, So, and Ye~\cite{zsy} introduced the {\em universal rigidity} of tensegrities. In terms of the PSD matrix completion problem,  a spherical tensegrity $(G,\Sigma, p)$ is defined to be universally rigid if ${\rm P}(G,\Sigma,\pi_G({\rm Gram}(p)))$ has a unique solution.
Thus Theorem~\ref{thm:unique_K32} and Theorem~\ref{thm:unique_K4} directly implies the following characterizations of universal rigidity.

\begin{corollary}
\label{thm:universal_K32}
Let $(G,\Sigma, p)$ be a spherical tensegrity in $\mathbb{S}^d$ for some positive integer $d$,
and let $c=\pi_G({\rm Gram}(p))$.
Suppose that $(G, \Sigma)$ is odd-$K_4$ minor and odd-$K_3^2$ minor free. 
Then the following holds:
\begin{itemize}
\setlength{\parskip}{0.1cm} 
 \setlength{\itemsep}{0.1cm} 
\item If ${\cal H}(G,\Sigma, c)$ has only one vertex, 
then $(G,p)$ is universally rigid with $\dim \spa p(V)=1$.
\item If ${\cal H}(G,\Sigma, c)$ has a hyperedge spanning all the vertices, then $(G,p)$ is universally rigid with $\dim \spa p(V)\leq 2$.
\item Otherwise, $(G,\sigma, c)$ is not universally rigid. 
\end{itemize} 
\end{corollary}

\begin{corollary}
\label{thm:universal_K4}
Let $(G,\Sigma, p)$ be a spherical tensegrity in $\mathbb{S}^d$ for some positive integer $d$,
and let $c=\pi_G({\rm Gram}(p))$.
Suppose that $(G, \Sigma)$ is odd-$K_4$ minor free.
Then the following holds:
\begin{itemize}
\setlength{\parskip}{0.1cm} 
 \setlength{\itemsep}{0.1cm}
\item If ${\cal H}(G,\Sigma, c)$ has only one vertex, 
then $(G,p)$ is universally rigid  with $\dim \spa p(V)=1$.
\item If ${\cal H}(G,\Sigma, c)$ has a hyperedge spanning all the vertices, then $(G,p)$ is universally rigid with $\dim \spa p(V)\leq 2$.
\item If ${\cal H}(G,\Sigma, c)$ forms a triangle, then $(G,p)$ is universally rigid with $\dim \spa p(V)\leq 3$.
\item Otherwise, $(G,\Sigma, c)$ is not universally rigid.
\end{itemize} 
\end{corollary}

Following a terminology introduced by Connelly~\cite{c,c14}, we say that 
a spherical tensegrity $(G,\Sigma, p)$ in $\mathbb{S}^{d-1}$ is {\em super stable} 
if $P(G, \Sigma, c)$ has a dual optimal solution $\omega$ of corank $d$ such that 
there is no nonzero symmetric matrix  $S$ of size $d$ satisfying  $p(i)^{\top} S p(j)=0$ for all  $ij\in V(G)\cup J_{\omega}$,
where  $c=\pi_G({\rm Gram}(p))$ and $J_{\omega}=\{ij\in E(G)\mid \omega(ij)\neq 0\}$.
It is known that super stability is a sufficient condition for universal rigidity~\cite{c,lv}.

We say that $(G, \Sigma, p)$ is {\em nondegenerate} if
$\sigma(ij) (p(i)\cdot p(j))\neq \sigma(ij)$ for every $ij\in E(G)$.
Combining the universal rigidity characterization given in Corollary~\ref{thm:universal_K4} with the strict complementarity shown in Theorem~\ref{thm:oddK4} we have the following. 
\begin{corollary}
Let $(G,\Sigma, p)$ be a spherical tensegrity in $\mathbb{S}^d$ for some positive integer $d$,
and let $c=\pi_G({\rm Gram}(p))$.
Suppose that $(G, \Sigma)$ is odd-$K_4$ minor free and $(G,\Sigma, p)$ is nondegenerate.
Then $(G, \Sigma, p)$ is universally rigid if and only if $(G, \Sigma, p)$ is super stable.
\end{corollary}

\section{Bounding Signed Colin de Verdi{\`e}re Parameter}\label{sec:colin}

For a signed graph $(G,\Sigma)$ with $n$ vertices, let $S(G, \Sigma)$ be the set of all symmetric matrices $A$ of size $n\times n$ such that 
\begin{itemize}
\setlength{\parskip}{0.1cm} 
 \setlength{\itemsep}{0.1cm} 
\item $A[i,j]<0$ if $i$ and $j$ are connected by only even edges, 
\item $A[i,j]>0$ if $i$ and $j$ are connected by only odd edges, 
\item $A[i,j]\in \mathbb{R}$ if $i$ and $j$ are connected by odd and even edges, 
\item $A[i,j]=0$ if $i$ and $j$ are non-adjacent.
\end{itemize}
A matrix $A\in S(G, \Sigma)$ is said to have the {\em Strong Arnold Property} (SAP for short)
if there is no nonzero symmetric $X$ of size $n$  satisfying  $AX=0$ and  $X[i,j]=0$ for all  $ij\in V(G)\cup E(G)$.
Arav et al.~\cite{ahlv13} defined the {\em signed Colin de Verdi{\`e}re parameter} $\nu(G,\Sigma)$ as the largest corank of any positive semidefinite matrix $A\in S(G, \Sigma)$ that has the SAP.
 In~\cite{ahlv16} they also gave the following characterization.
\begin{theorem}[Arav et al.~\cite{ahlv16}]
\label{thm:colin}
For a signed graph $(G,\Sigma)$, $\nu(G,\Sigma)\leq 2$ if and only if $(G,\Sigma)$
is odd-$K_4$ minor free and odd-$K_3^2$ minor free.
\end{theorem} 

It is not difficult to see that the family of signed graphs with $\nu(G,\Sigma)\leq k$ is minor closed for each $k$. Hence  one direction of Theorem~\ref{thm:colin} follows by showing that both odd-$K_4$ and odd-$K_3^2$ have the parameter value more than two, 
which can be done by explicitly constructing a matrix $A$ of corank three for each graph~\cite{ahlv13}.
The main result of \cite{ahlv16} is to prove the other direction (sufficiency), 
and here we shall show how to derive the sufficiency of Theorem~\ref{thm:colin} from Theorem~\ref{thm:universal_K32}.

As observed in \cite[Theorem~5.2]{lv}, if $A\in S(G,\Sigma)$ is positive semidefinite, 
the condition for SAP is equivalent to the following condition:
there is no nonzero symmetric matrix $S$ of size $d$ satisfying  $p(i)^{\top} S p(j)=0$ for all  $ij\in V(G)\cup E(G)$,
where $p$ denotes a map $p: V(G)\rightarrow \mathbb{S}^{d-1}$ such that  
${\rm Gram}(p)$ and $A$ satisfies the  strict complementarity condition.
Therefore, if $A\in S(G,\Sigma)$ satisfies the SAP,  $A$ gives a certificate for the super stability of $(G,\Sigma,p)$. 
Hence $\nu(G,\Sigma)$ is upper bounded by the maximum $d$ such that 
there is a super stable spherical tensegrity $(G,\Sigma, p)$ in $\mathbb{S}^{d-1}$.
Since the super stability is a sufficient condition for universal rigidity~\cite{c,lv}, 
this further implies that
$\nu(G,\Sigma)$ is upper bounded by maximum $d$ such that 
there is a universally rigid spherical tensegrity $(G,\Sigma, p)$ in $\mathbb{S}^{d-1}$
(with the condition that $p(V)$ linearly spans $\mathbb{R}^d$).
Theorem~\ref{thm:universal_K32} says that, if $(G,\Sigma)$ is odd-$K_4$ minor free and odd-$K_3^2$ minor free, then $(G, \Sigma)$ cannot be universally rigid in $\mathbb{S}^{d-1}$ for any $d\geq 3$.
We thus obtain $\nu(G,\Sigma)\leq 2$.

%
%
%

\section{Conclusion}
We conclude the paper by listing open problems.
An interesting open problem  is to establish the signed generalization of a characterization of ${\cal E}(G)$ by Barrett, Johnson, and Loewy~\cite{bjl} in terms of the metric inequalities and the so-called clique conditions.  

We have remarked that, if $(G, \Sigma)$ is odd-$K_4$ minor, then one can compute a completion with rank at most three in polynomial time in the real RAM model. Finding a larger class of signed graphs whose rank-constrained completion problems are tractable would be an important question. See~\cite{b07,bc07,lv14} for the progress on this question for unsigned graphs.

Going beyond the odd-$K_4$ minor free graphs would be a challenging problem in every aspect of this paper. This will be also related to the polynomial-time solvability of the PSD matrix completion problem in the real number model. See, e.g., \cite{l97,l98,l00}.

\section*{Acknowledgement}
The author would like to thank  Monique Laurent for fruitful discussions on the topic of this paper.

This work was supported by JSPS Postdoctoral Fellowships for Research Abroad, 
JSPS Grant-in-Aid for Young Scientist (B) 15K15942, 
and JSPS Grant-in-Aid for Scientific Research (C) 15KT0109.

\end{document}